\documentclass[11pt,reqno]{amsart}
\usepackage{amsthm,amssymb,amsmath}

\usepackage[T1]{fontenc}
\usepackage{lmodern}

\newtheorem{theorem}{Theorem}
\newtheorem*{theorem*}{Theorem}

\newtheorem{proposition}{Proposition}

\theoremstyle{definition}

\newtheorem{definition}{\sc Definition}
\newtheorem*{definition*}{\sc Definition}

\newtheorem{remark}{Remark}
\newtheorem*{remark*}{Remark}

\newtheorem*{example*}{\bf Example}

\newcommand{\loc}{{\rm loc}}
\newcommand{\supp}{{\rm sprt\,}}

\expandafter\def\expandafter\normalsize\expandafter{%
    \normalsize
    \setlength\abovedisplayshortskip{8pt}
    \setlength\belowdisplayshortskip{8pt}
}

\begin{document}

\title[Divergence-free drifts]{On divergence-free  (form-bounded type) drifts}

\dedicatory{Dedicated to Jerry Goldstein on the occasion of his 80th birthday.}

\begin{abstract}
We develop regularity theory for elliptic Kolmogorov operator  with divergence-free drift in a large class  (or, more generally, drift having singular divergence). A key step in our proofs is ``Caccioppoli's iterations'', used in addition to the classical De Giorgi's iterations and Moser's method.
\end{abstract}

\author{Damir Kinzebulatov and Reihaneh Vafadar}

\email{damir.kinzebulatov@mat.ulaval.ca}
\email{reihaneh.vafadar-seyedi-nasl.1@ulaval.ca}

\address{Universit\'{e} Laval, D\'{e}partement de math\'{e}matiques et de statistique, Qu\'{e}bec, QC, Canada}



\thanks{D.K.\,is supported by the NSERC (grant RGPIN-2017-05567). R.V.\,is supported by the ISM, Montr\'{e}al.}

\maketitle

\section{Introduction}

\textbf{1.~}This paper is motivated by the following question: what minimal assumptions on a vector field $b$ on $\mathbb R^d$ ($d \geq 3$) ensure that the ``classical'' regularity theory of equations
\begin{equation}
\label{eq1}
(-\Delta + b \cdot \nabla)u=0
\end{equation}
and
\begin{equation}
\label{eq2}
(-\nabla \cdot a \cdot \nabla + b \cdot \nabla)u=0
\end{equation}
is still valid? The matrix $a$ is assumed to be measurable and uniformly elliptic, i.e.
\begin{equation}
\tag{$H_{\sigma,\xi}$}
\sigma I \leq a \leq \xi I \text{ a.e. on } \mathbb R^d, \quad 0<\sigma \leq \xi<\infty,
\end{equation}
and the vector field $b$ is assumed to be divergence-free or, more generally, to have divergence in $L^1_{\loc}\equiv L^1_{\loc}(\mathbb R^d)$.
In the former case, the elliptic equation \eqref{eq1} can be viewed as a proxy to the corresponding parabolic equation with a time-inhomogeneous divergence-free vector field. This equation plays an important role in hydrodynamics, e.g.\,as the equation behind the passive tracer SDE where the drift $b$ is the velocity field obtained by solving 3D Navier-Stokes equations, see e.g.\,\cite{MK}.

One can prove e.g.\,local boundedness of weak solutions to the corresponding parabolic equation and a Harnack-type inequality requiring from a divergence-free $b$ only $\|b\|_p<\infty$ for a $p>\frac{d}{2}$, see Zhang \cite{Z}, Nazarov-Uraltseva \cite{NU}. (Here and below $\|f\|_p:=(\int_{\mathbb R^d} |f|^pdx)^{1/p}$.)
 However, to have a classical regularity theory, including H\"{o}lder continuity of solutions to \eqref{eq2}, one needs $p=d$. Informally, one arrives at $p=d$ by requiring that rescaling the operator leaves invariant the norm of $b$ (neither decreases it, otherwise the singularities of $b$ would be too weak and easy to deal with, nor increases it, otherwise this can destroy the continuity of even bounded solutions to \eqref{eq1}, see Filonov \cite{F}).
The present paper deals with scaling-invariant conditions on $b$.

Speaking of the choice of $\|b\|_d$ as a measure of singularity of $b$, it is convenient as long as one agrees that the task of verifying $\|b\|_d<\infty$ is, in principle, elementary. Nevertheless, this choice is somewhat arbitrary since it largely ignores, beyond the scaling considerations, the operator behind the equation. Much more broad  conditions on $b$ are possible:

\begin{definition}
A vector field $b:\mathbb R^d \rightarrow \mathbb R^d$ is said to be multiplicatively form-bounded if $|b| \in L^1_{\loc}$ and there exists a constant $0<\delta<\infty$ such that
$$
\langle |b|\varphi,\varphi\rangle \leq \delta\|\nabla \varphi\|_2\|\varphi\|_2 + c_\delta\|\varphi\|_2^2 \quad \forall\,\varphi \in W^{1,2}
$$
for some $c_\delta<\infty$ (written as $b \in \mathbf{MF}_\delta$). The constant $\delta$ is called a weak form-bound of $b$.

Here and below
$$
\langle g\rangle:=\int_{\mathbb R^d}gdx, \quad \langle f,g\rangle:=\langle fg\rangle
$$
(all functions considered below are real-valued), and $W^{1,p}=W^{1,p}(\mathbb R^d)$ denotes the usual Sobolev space.
\end{definition}

There is a well developed machinery that allows to verify inclusion $b \in \mathbf{MF}_\delta$. For instance, the following classes of vector fields $b$, defined in elementary terms, are contained in $\mathbf{MF}_\delta$:

1) $|b| \in L^d$, in which case $\delta$ can be chosen arbitrarily small;

2) $|b|$ in the weak $L^d$ class;

3) More generally, $|b|^2$ is in the Chang-Wilson-Wolff class \cite{CWW}, i.e.\,
$$
\sup_Q \frac{1}{|Q|}\int_Q |b|^2\, l(Q)^2 \gamma\big(|b|^2\,l(Q)^2 \big) dx<\infty,
$$
where
$\gamma:[0,\infty[ \rightarrow [1,\infty[$ is an increasing function such that
$
\int_1^\infty \frac{dt}{t\gamma(t)}<\infty.
$
The Chang-Wilson-Wolff class contains the Campanato-Morrey class
$$|b| \in L^{2s}_\loc \text{ for some $s>1$ and } \bigl(\frac{1}{|Q|}\int_Q |b(x)|^{2s} dx \bigr)^{\frac{1}{2s}} \leq c_s l(Q)^{-1} \text{ for all cubes $Q$.}$$

4) $|b|$ in the Campanato-Morrey class 
\begin{equation}
\label{cm}
|b| \in L^{s}_\loc \text{ for some $s>1$ and }  \bigl(\frac{1}{|Q|}\int_Q |b(x)|^{s} dx \bigr)^{\frac{1}{s}} \leq c_s l(Q)^{-1}  \text{ for all cubes $Q$}.
\end{equation}

5) $b$ in the Kato class of vector fields $\mathbf{K}^{d+1}_\delta$, i.e.\,$|b| \in L^1_\loc$ and $\|(\lambda-\Delta)^{-\frac{1}{2}}|b|\|_\infty \leq \delta$ for a $\lambda>0$.

In 2)-5) the value of $\delta$ is proportional to the norm of $|b|$ in respective classes.

\medskip

Below we show that $\mathbf{MF}_\delta$ contains the standard class of form-bounded vector fields $\mathbf{F}_\delta$ (see Definition \ref{fb_def}). In turn, $\mathbf{F}_\delta$ contains classes 1)-3) (we note in passing that the Chang-Wilson-Wolff class was born out of attempts by many authors to obtain a necessary and sufficient condition for ``$b \in \mathbf{F}_\delta$'' in elementary terms, see references in \cite{CWW}). The inclusions of 4) and 5) in $\mathbf{MF}_\delta$ follow from \eqref{incl_wfb} and the fact that vector fields in the Campanato-Morrey class \eqref{cm} or the Kato class are weakly form-bounded, see below. 

The class $\mathbf{MF}_\delta$ neither is contained in, nor contains another well-known class of divergence-free vector fields $\mathbf{ BMO}^{-1}$ (i.e.\,$b=\nabla F$ for skew-symmetric matrix-valued function $F$ with entries in the space ${\rm BMO}={\rm BMO}(\mathbb R^d)$ of functions of bounded mean oscillation, see Definition \ref{bmo_def}). See, however, Remark \ref{comb_rem} about combining $\mathbf{MF}_\delta$ and $\mathbf{ BMO}^{-1}$.

\medskip

In the present paper we develop De Giorgi's approach to the regularity theory of Kolmogorov operator
\begin{equation}
\label{kolm}
-\nabla \cdot a \cdot \nabla + b \cdot \nabla,
\end{equation}
 with $b \in \mathbf{MF}_\delta$, $\delta<\infty$ and ${\rm div\,}b$ satisfying some broad assumptions (see \eqref{nu}, \eqref{nu2} below). The multiplicative form-boundedness was introduced by Sem\"{e}nov \cite{S} as a condition providing a priori two-sided Gaussian bound on the heat kernel of \eqref{kolm}
and hence its a priori H\"{o}lder continuity, assuming $b$ is divergence-free. His proof of the upper Gaussian bound used Moser's method. The proof of the lower Gaussian bound in \cite {S} required a deep  modification of Nash's method. The reason is that when dealing with multiplicative form-boundedness, one cannot use one of the key instruments  in the analysis of PDEs: the quadratic inequality.
Our motivation, beyond the desire to arrive at a priori H\"{o}lder continuity of solutions to \eqref{eq2} using a somewhat simpler argument, and curiosity (in fact, the proof of Caccioppoli's inequality for multiplicatively form-bounded $b$ turned out to be rather interesting, see below), is driven by the following two goals not addressed by the other methods:

\medskip

-- \textit{A posteriori solution theory for Dirichlet problem for \eqref{eq2} and a posteriori Harnack inequality.}

\smallskip

In Theorem \ref{thm3} we prove approximation uniqueness of solution to Dirichlet problem for \eqref{eq2} with multiplicatively form-bounded $b$, going beyond the borderline case $|b| \in L^2_{\loc}$. (The approximation uniqueness means uniqueness among weak solutions that can be constructed via an approximation procedure.) The proof uses higher integrability of the gradient of solution.

\medskip

-- \textit{The minimal assumptions on the divergence of $b$.}

\smallskip

Physical applications require one to treat singular ${\rm div\,}b$.
We allow in this paper $({\rm div\,}b)_\pm$ in the class of form-bounded potentials, i.e.\,$({\rm div\,}b)_\pm \in L^1_{\loc}$ and
\begin{equation}
\label{nu}
\langle ({\rm div\,}b)_+\varphi,\varphi\rangle \leq \nu_+ \|\nabla \varphi\|_2^2 + c_{\nu_+}\|\varphi\|_2^2, \quad \nu_+<2\sigma,
\end{equation}
\begin{equation}
\label{nu2}
\langle ({\rm div\,}b)_-\varphi,\varphi\rangle \leq \nu_- \|\nabla \varphi\|_2^2 + c_{\nu_-}\|\varphi\|_2^2, \quad \nu_-<\infty.
\end{equation}
for some $c_{\nu_\pm}<\infty$, for all $\varphi \in W^{1,2}$. (Throughout the paper, given a function $f$, we denote by $f_\pm$ its positive/negative part.)
For instance, potentials in the weak $L^{\frac{d}{2}}$ class are form-bounded, but there also exist form-bounded potentials $\not\in L^{1+\varepsilon}_{\loc}$, for arbitrarily fixed $\varepsilon>0$. 

Earlier, Kinzebulatov-Sem\"{e}nov \cite{KiS3} established a priori two-sided Gaussian heat kernel bounds for \eqref{kolm} for $b \in \mathbf{MF}_\delta$, $\delta<\infty$ with ${\rm div\,}b$ in the Kato class of potentials $\mathbf{K}^d_\nu$, a proper subclass of \eqref{nu}, \eqref{nu2}. Example of a vector field $b$ such that ${\rm div\,}b$ is form-bounded but is not in the Kato class is given by e.g.\,
\begin{equation}
\label{vf}
b(x)=\kappa_+|x-x_+|^{-2}(x-x_+) - \kappa_-|x-x_-|^{-2}(x-x_-) \qquad (\kappa_\pm > 0),
\end{equation}
 where $x_\pm \in \mathbb R^d$ are fixed. Here ${\rm div\,}b=\kappa_+(d-2)|x-x_+|^{-2} - \kappa_-(d-2)|x-x_-|^{-2}$ satisfies \eqref{nu}, \eqref{nu2} with $\nu_\pm=\frac{4\kappa_\pm}{d-2}$, $c_{\nu_\pm}=0$ by Hardy's inequality. 
This vector field indeed destroys two-sided Gaussian bound, see discussion in \cite{KiS3}. 
At the level of the corresponding to \eqref{eq1}, \eqref{vf} SDE $dX_t=-b(X_t)dt+\sqrt{2}dB_t$ the first term in \eqref{vf}  forces the diffusion process $X_t$ to approach $x_+$, while the second term pushes $X_t$ away from $x_-$, i.e.\,taking into account $({\rm div\,}b)_\pm$ allows to model attraction/repulsion phenomena. 

The form-boundedness of $({\rm div\,}b)_+$ seems to be the maximal possible assumption on ${\rm div\,}b$ providing a ``classical'' regularity theory of \eqref{eq2}.
If $b$ belongs to a smaller class $\mathbf{F}_\delta$, then no assumption on the negative part of ${\rm div\,}b$ is needed, see Hara \cite{H}.

\medskip

\textbf{2.~}The starting point of De Giorgi's method is Caccioppoli's inequality (cf.\,Proposition \ref{c_prop}). Let us outline its derivation and describe ensuing difficulties when dealing with multiplicatively form-bounded $b$. 
First, we introduce some notations used throughout the paper. Let $B_r(x)$ denote the open ball in $\mathbb R^d$ of radius $r$ centered at $x$. Put $$B_r:=B_r(0).$$
Denote by $(f)_{B}$ the average of function $f$ over a ball (or some other set) $B$:
$$
(f)_{B}:=\frac{1}{|B|}\langle f\mathbf{1}_{B}\rangle, \quad |B|:={\rm Vol}\,B.
$$
Now, to prove Caccioppoli's inequality, one multiplies equation \eqref{eq2} by $u\eta$, where $u \in W^{1,2}$ is a weak solution to the equation and $\eta \in C_c^\infty(B_R)$, $R \leq 1$ is a $[0,1]$-valued function which is identically $1$ on a concentric ball of smaller radius $r<R$ and satisfies $|\nabla \eta|^2/\eta \leq c(R-r)^{-2}\mathbf{1}_{B_R}$. Integrating and using $a \in (H_{\sigma,\xi})$, one obtains right away
\begin{equation*}
\sigma \langle |\nabla u|^2\eta\rangle \leq \langle a \cdot \nabla u,u,\nabla \eta \rangle + |\langle b \cdot \nabla u,u\eta\rangle|
\end{equation*}
so, applying quadratic inequality, one has
\begin{equation}
\label{dg}
\sigma \langle |\nabla u|^2\eta\rangle \leq \epsilon \sigma\langle |\nabla u|^2\eta\rangle + \frac{\sigma}{4\epsilon}\langle u^2\frac{|\nabla \eta|^2}{\eta}\rangle + |\langle b \cdot \nabla u,u\eta\rangle| \qquad (\epsilon>0).
\end{equation}
Then, in particular,
\begin{equation}
\label{dg2}
\sigma(1-\epsilon)\langle |\nabla u|^2\mathbf{1}_{B_r}\rangle \leq \frac{c\sigma }{4\epsilon (R-r)^2}\langle u^2\mathbf{1}_{B_R}\rangle + |\langle b \cdot \nabla u,u\eta\rangle|.
\end{equation}
Thus, in the LHS of \eqref{dg2} one obtains extra information about the regularity of $u$ but on a smaller set $B_r$, provided that one can control $\langle b \cdot \nabla u,u\eta\rangle$:

\smallskip

(a) If $b \in \mathbf{F}_\delta$, $\delta<\sigma^2$ (no assumptions on ${\rm div\,}b$), then one has, using quadratic inequality
\begin{align*}
|\langle b \cdot \nabla u,u\eta\rangle| &  
\leq \alpha \langle |\nabla u|^2\eta \rangle  + \frac{1}{4\alpha}\langle |b|^2, u^2\eta \rangle \quad (\alpha>0).
\end{align*}
Now, applying $b \in \mathbf{F}_\delta$, minimizing in $\alpha$
and substituting the result in \eqref{dg} (with $\epsilon$ chosen sufficiently small), one obtains the Caccioppoli inequality.

\medskip

(b) If $b=\nabla F \in \mathbf{ BMO}^{-1}$, then one obtains, using the divergence theorem,
\begin{align*}
|\langle b \cdot \nabla u,u\eta\rangle| \leq \gamma \big\langle |\nabla u|^2\eta \big\rangle + \frac{1}{\gamma}\big\langle |F-(F)_{B_R}|^2,u^2\frac{|\nabla \eta|^2}{\eta} \big\rangle \quad (\gamma>0).
\end{align*}
The latter yields a Caccioppoli-type inequality, where the term containing $|F-(F)_{B_R}|^2$ is handled using e.g.\,the John-Nirenberg inequality. We refer to \cite{H} for details, see also \cite{SSSZ,Zh}. 

\medskip

(c) If $b \in \mathbf{MF}_\delta$, $\delta<\infty$ with, say, ${\rm div\,}b=0$, then one has
\begin{align}
|\langle b \cdot \nabla u,u\eta\rangle| & =\frac{1}{2}|\langle bu,u\nabla \eta\rangle| \leq \frac{1}{2}\langle |b|u,u|\nabla \eta|\rangle \notag \\
& \leq \frac{\delta}{2} \bigg(\|(\nabla u)\sqrt{|\nabla \eta|}\|_2 + \|u \nabla \sqrt{|\nabla \eta|}\|_2 \bigg) \, \|u\sqrt{|\nabla \eta|}\|_2 + \frac{c_\delta}{2} \|u\sqrt{|\nabla \eta|}\|^2_2 \label{m_ineq}
\end{align}
(see the proof of Proposition \ref{c_prop}). The estimates $$\sqrt{|\nabla \eta|} \leq c(R-r)^{-\frac{1}{2}}\mathbf{1}_{B_R}, \quad |\nabla \sqrt{|\nabla \eta|}| \leq c (R-r)^{-\frac{3}{2}}\mathbf{1}_{B_R}$$ present no problem. The difficulty is in the term $\|(\nabla u)\sqrt{|\nabla \eta|}\|_2$ in the RHS of \eqref{m_ineq}: one has $\nabla u$ and $\nabla \eta$ at the same time. Thus, one cannot simply transition this term to the LHS of \eqref{dg} as in (a). Furthermore, estimating $|\nabla \eta| \leq c(R-r)^{-1}\mathbf{1}_{B_R}$ in $\|(\nabla u)\sqrt{|\nabla \eta|}\|_2$  one obtains the norm of $\nabla u$ over a larger set than in the LHS of \eqref{dg2}. 
Nevertheless, it turns out that one can iterate the resulting from \eqref{dg2} and \eqref{m_ineq} inequality over balls of radii between $r$ and $R$, which leads to the sought Caccioppoli's inequality, see the proof of Proposition \ref{c_prop}.
This iteration procedure (``Caccioppoli's iterations'') is also used in Moser's method in the proof of Proposition \ref{pp_prop}, although in a slightly more sophisticated form.

\medskip

\textbf{3.~}Let us now recall the definitions of the class of form-bounded vector fields and the class $\mathbf{ BMO}^{-1}$.

\begin{definition}
\label{fb_def}
A vector field $b:\mathbb R^d \rightarrow \mathbb R^d$ is said to be form-bounded if $|b| \in L^2_{\loc}$ and there exists $\delta>0$ such that
$$
\langle |b|^2 \varphi,\varphi\rangle \leq \delta \|\nabla \varphi\|_2^2 + c_\delta \|\varphi\|_2^2 \quad \forall \varphi \in W^{1,2}
$$
for some constant $c_\delta$ (written as $b \in \mathbf{F}_\delta$). No conditions on ${\rm div\,}b$ are imposed.
\end{definition}

\begin{definition}
\label{bmo_def}
A divergence-free distributional vector field $b \in [\mathcal S']^d$ is in the class $\mathbf{ BMO}^{-1}$ if
\begin{equation*}
b=\nabla F \quad \text{ i.e. } b_k=\sum_{i=1}^d \nabla_i F_{ik}, \quad 1 \leq k \leq d,
\end{equation*}
for  matrix $F$ with entries $F_{ik}=-F_{ki} \in {\rm BMO}$. (Recall that $F_{ik} \in {\rm BMO}$ means that $F_{ik} \in L^1_{\loc}$ and
$$
\|F_{ik}\|_{\rm BMO}:=\sup_{Q}\frac{1}{|Q|}\int_{Q}|F-(F)_Q|dx<\infty,
$$
where the supremum taken over all cubes $Q \subset \mathbb R^d$ with sides parallel to the axes.)

\end{definition}

The class of form-bounded vector fields $\mathbf{F}_\delta$ contains $L^d$ class, the weak $L^d$ class, the Campanato-Morrey class and the Chang-Wilson-Wolff class, see \cite{KiS1}.  It provides a posteriori Harnack inequality and H\"{o}lder continuity of solutions to \eqref{eq2}, as long as $\delta<\sigma^2$, see Hara \cite{H}, and weak well-posedness of the corresponding to \eqref{eq1} SDE, as long as form-bound $\delta<c_d$ for  a certain explicit constant $c_d<1$, see Kinzebulatov-Sem\"{e}nov \cite{KiS2} and Kinzebulatov-Madou \cite{KiM}. It also provides a posteriori upper and/or (depending on ${\rm div\,}b$) lower  Gaussian bound on the heat kernel of Kolmogorov operator \eqref{kolm}, see Kinzebulatov-Sem\"{e}nov \cite{KiS3}. This list of results involving form-bounded drift is far from being exhaustive. See, in particular, Zhang \cite{Z} regarding elements of the regularity theory of \eqref{eq2} under supercritical (in the sense of scaling) form-boundedness type assumption on $b$.

The  quantitative role of form-bound $\delta$ in the theory of Kolmogorov operator was recognized by Kovalenko-Sem\"{e}nov \cite{KS} who proved $W^{1,p}$ estimates on solutions to \eqref{eq1} with $b \in \mathbf{F}_\delta$, where the interval of admissible $p$ expands to $[2,\infty[$ as $\delta$ is taken closer and closer to zero. These estimates, with $p$ sufficiently large, allow to construct the corresponding to \eqref{eq1} Feller semigroup. 

The class $\mathbf{ BMO}^{-1}$ contains divergence-free vector fields with entries e.g.\,in the Campanato-Morrey class \eqref{cm}, and it also contains some singular distribution vector fields \cite{KT}, which, obviously, can not be multiplicatively form-bounded. A weak solution theory of the Dirichlet problem for \eqref{eq1} with $b \in \mathbf{ BMO}^{-1}$ was developed by Zhikov \cite{Zh}. This class furthermore provides a posteriori  Harnack inequality and H\"{o}lder continuity of solutions to the parabolic counterpart of \eqref{eq2}, see Friedlander-Vicol \cite{FV}, Seregin-Silvestre-\v{S}verak-Zlato\v{s} \cite{SSSZ}, and a posteriori two-sided Gaussian bound on Kolmogorov operator \eqref{kolm}, see Qian-Xi \cite{QX}. 
Earlier, a subclass of $\mathbf{ BMO}^{-1}$ consisting of vector fields $b=\nabla B$ for skew-symmetric $B$ with entries in $L^\infty$ was considered by Osada \cite{O}.

\medskip

\textbf{4.~}We note that for $b=b_1+b_2$, where $b_1 \in \mathbf{F}_\delta$, $b_2 \in \mathbf{ BMO}^{-1}$ one has the KLMN Theorem in $L^2$ via the estimate 
$$
|\langle b \cdot \nabla u,v\rangle| \leq (\sqrt{\delta} + \|F\|_{\rm BMO})\|\nabla u\|_{2}(\|\nabla v\|_{2}+c\|v\|_2), \quad c=\frac{\sqrt{c_\delta}}{\sqrt{\delta}}
$$
(for $b_2$, using the compensated compactness estimate of \cite{CLMS}). The KLMN Theorem provides an a posteriori solution theory for \eqref{kolm} in $L^2$. This settles for such $b$'s the problem of a posteriori Harnack inequality. On the other hand, no analogues of the KLMN Theorem are known so far for $b \in \mathbf{MF}_\delta$ except in some special cases (cf.\,the first comment in Section \ref{disc_sect}).

Regarding the inclusion $\mathbf{F}_{\delta^2} \subset \mathbf{MF}_\delta$ mentioned above, in fact,
 a much stronger statement is true: $\mathbf{MF}_\delta$ contains the class of weakly form-bounded vector fields, that is, $|b| \in L^1_{\loc} $ and
\begin{equation}
\tag{$\mathbf{F}_\delta^{\scriptscriptstyle \frac{1}{2}}$}
\||b|^{\frac{1}{2}}\varphi \|^2_2 \leq \delta \|(\lambda-\Delta)^{\frac{1}{4}}\varphi \|^2_2 \quad \forall\,\varphi \in\mathcal W^{\frac{1}{2},2}
\end{equation}
for some $\lambda>0$ (which, in turn, contains  $\mathbf{F}_{\delta^2}$ with $c_\delta=\lambda\delta$, as is evident from the Heinz-Kato inequality). Here and below, $\mathcal W^{\alpha,p}$ denoted the Bessel potential space. Indeed, 
\begin{align}
\label{incl_wfb}
\langle|b|\varphi,\varphi\rangle & \leq \delta \langle (\lambda-\Delta)^{\frac{1}{2}}\varphi,\varphi \rangle \leq \delta \|(\lambda-\Delta)^{\frac{1}{2}}\varphi\|_2\|\varphi\|_2 \\
& \leq \delta \|\nabla \varphi\|_2\|\varphi\|_2 + \lambda\|\varphi\|_2^2 \quad \Rightarrow \quad b \in \mathbf{MF}_\delta.
\end{align}
Thus, compared to the standard form-boundedness, the multiplicative form-boundedness allows to gain twice in the a priori summability requirement on the vector field, i.e.\,$|b| \in L^1_{\loc}$ instead of $|b| \in L^2_{\loc}$. 

Note that $\mathbf{F}_\delta^{\scriptscriptstyle 1/2}$ contains the Kato class $\mathbf{K}^{d+1}_\delta$ (e.g.\,by interpolation) and the Campanato-Morrey class \eqref{cm} \cite{A}, hence for every $\varepsilon>0$ one can find weakly form-bounded vector fields $b$ with $|b| \not\in L^{1+\varepsilon}_{\loc}$.

See further discussion in Section \ref{disc_sect}.

\bigskip

\section{Main results} 

Our first result concerns a priori estimates for equation \eqref{eq2}, i.e.\,the coefficients will be assumed to be smooth. 
Let $\Omega \subset \mathbb R^d$ be a bounded Lipschitz domain, $d \geq 3$.

\begin{definition}
We call a constant generic 
if it depends only on $d$, $\Omega$, $\sigma$, $\xi$, $\delta$, $c_\delta$, $\nu_\pm$ and $c_{\nu_\pm}$.
\end{definition}

\begin{theorem}
\label{thm1}
Let $a \in (H_{\sigma,\xi})$ and $b \in \mathbf{MF}:=\cup_{\delta>0}\mathbf{MF}_\delta$. Assume that $a$,  $b$ are bounded smooth. Also, assume that ${\rm div\,}b={\rm div\,}b_+ - {\rm div\,}b_-$ for bounded smooth  ${\rm div\,}b_\pm \geq 0$ satisfying form-boundedness conditions \eqref{nu}, \eqref{nu2}.
Let $B_{2R}(x) \subset \Omega$, and let $u$ be a solution to \eqref{eq2} in $B_{R}(x)$. Then 
\begin{equation}
\label{grad}
\|\nabla u\|_{L^p(B_{\frac{R}{2}}(x))}\leq C_0
\end{equation}
 for some generic $p>2$ and $C_0<\infty$. Moreover, if solution $u \geq 0$ in $B_R(x)$, then it satisfies 
the Harnack inequality 
\begin{equation}
\label{harnack}
\sup_{B_{R/2}(x)}u \leq C \inf_{B_{R/2}(x)}u
\end{equation}
with generic constant $C$,
and is H\"{o}lder continuous:
\begin{equation}
\label{holder}
{\rm osc}_{B_r(x)}u \leq K\left(\frac{r}{R}\right)^{\gamma}{\rm osc}_{B_{R/2}}u, \quad 0<r<\frac{R}{2}
\end{equation}
for some generic constants $K$ and $0<\gamma<1$.
\end{theorem}

Of course, a key point of Theorem \ref{thm1} is that constants $C_0$, $C$, $K$, $\gamma$ do not depend on the smoothness of $a$, $b$, ${\rm div\,}b$ or the $L^\infty$ norms of the last two.

Theorem \ref{thm1} uses the standard De Giorgi's iterations and Moser's method, with the addition of ``Caccioppoli's iterations'' in Propositions \ref{c_prop} and \ref{pp_prop}.

\begin{remark}
\label{comb_rem}

Theorem \ref{thm1} extends to vector fields
\begin{equation}
\label{bb}
b=b_1+b_2, \quad b_1 \in \mathbf{MF}, \quad b_2 \in \mathbf{BMO}^{-1}
\end{equation}
where ${\rm div\,}b_1$ satisfies the assumptions of Theorem \ref{thm1}, see Remark \ref{bb_rem} in the end of the proof. See also comment 4 in Section \ref{disc_sect} regarding Nash's method.
\end{remark}

We now turn to the question of weak well-posedness of Dirichlet problem
\begin{equation}
\label{d0}
\left\{
\begin{array}{l}
(-\nabla \cdot a \cdot \nabla + b \cdot \nabla)u=0 \quad \text{ in } \Omega  \\
u-g \in W_0^{1,2}(\Omega),
\end{array}
\right.
\end{equation}
where $a \in (H_{\sigma,\xi})$ and $b \in \mathbf{MF}$ are assumed to be only measurable, ${\rm div\,}b \in L^1_{\loc}$ satisfies \eqref{nu}, \eqref{nu2}, and
\begin{equation}
\label{g}
g \in L^\infty, \quad \|g\|_{W^{2,2}}<\infty.
\end{equation}

\begin{definition}  
\label{weak_sol_def}
We say that $u \in W^{1,2}_{\loc}(\Omega)$ is a weak solution to equation
\begin{equation}
\label{eq2_}
(-\nabla \cdot a \cdot \nabla + b \cdot \nabla)u=0 \quad \text{ in } \Omega 
\end{equation}
 if 

{\rm (\textit{i})} $bu \in [L^1_{\loc}(\Omega)]^d$, $({\rm div\,}b)u \in L^1_{\loc}(\Omega)$ and

{\rm (\textit{ii})} identity
$$
\langle a \cdot \nabla u,\nabla \varphi\rangle - \langle bu,\nabla \varphi\rangle - \langle ({\rm div\,}b)u,\varphi\rangle =0, \quad \forall\varphi \in C_c^\infty(\Omega)
$$ 
holds.
\end{definition}

If $u$ is locally bounded, then (\textit{i}) is trivially satisfied.

\medskip

We will construct a weak solution to \eqref{d0} via an approximation procedure.
Let us fix
$C^\infty$ smooth bounded $b_n$ such that
\begin{equation}
\label{b_n}
b_n \in \mathbf{MF}_\delta \;\;\text{with the same $c_\delta$ (so, independent of $n$)},  \quad 
b_n \rightarrow b \quad \text{ in } [L^1_{\loc}]^d, 
\end{equation}
\begin{equation}
\label{div_n}
\begin{array}{c}
{\rm div\,}b_n={\rm div\,}b_{n,+} - {\rm div\,}b_{n,-}, \\
0 \leq {\rm div\,}b_{n,\pm} \in C^\infty \cap L^\infty \text{ satisfy \eqref{nu}, \eqref{nu2} with the same  $\nu_\pm$, $c_{\nu_\pm}$ (so, independent of $n$)}, 
\end{array}
\end{equation}
\begin{equation}
\label{div_n2}
  {\rm div\,}b_{n,\pm} \rightarrow ({\rm div\,}b)_{\pm} \text{ in } L^1_{\loc}.
\end{equation}
We emphasize that ${\rm div\,}b_{n,\pm}$ above is any pair of non-negative functions such that identity ${\rm div\,}b_n={\rm div\,}b_{n,+} - {\rm div\,}b_{n,-}$ holds.
We discuss a construction of such $b_n$ in Section \ref{approx_sect}.

We fix bounded smooth $g_n$ such that
\begin{equation}
\label{g_n}
g_n \rightarrow g \quad \text { weakly in } W^{2,2}_{\loc}, \quad \|g_n\|_\infty \leq \|g\|_\infty.
\end{equation}
Let us also fix $C^\infty$ smooth $a_n \in (H_{\sigma,\xi})$,
\begin{equation}
\label{a_n}
a_n \rightarrow a \quad \text { in } [L^1_{\loc}]^{d \times d}.
\end{equation}

\begin{theorem}
\label{thm2}
Let $a \in (H_{\sigma,\xi})$, $b \in \mathbf{MF}$ with ${\rm div\,}b \in L^1_{\loc}$ satisfying \eqref{nu}, \eqref{nu2}. Let $(a_n,b_n,g_n)$ satisfy \eqref{b_n}-\eqref{a_n}.
Then solutions $u_n$ to the Dirichlet problems 
\begin{equation}
\label{d2}
\left\{
\begin{array}{l}
(-\nabla \cdot a_n \cdot \nabla + b_n \cdot \nabla)u_n=0 \quad \text{ in } \Omega  \\
u_n=g_n \text{ on } \partial \Omega
\end{array}
\right.
\end{equation}
converge weakly in $W^{1,2}_{\loc}(\Omega)$ as $n \rightarrow \infty$, possibly after passing to a subsequence, to a  weak solution to Dirichlet problem \eqref{d0}. This weak solution is bounded, satisfies the gradient estimate \eqref{grad} and, if $g \geq 0$, satisfies the Harnack inequality \eqref{harnack} and is  H\"{o}lder continuous, cf.\,\eqref{holder}. 
\end{theorem}

The proof of convergence/existence part of Theorem \ref{thm2} requires some care, since the moment one puts problem \eqref{d0} for $a=a_n$, $b=b_n$, $g=g_n$ in the form
\begin{equation*}
\left\{
\begin{array}{l}
(-\nabla \cdot a \cdot \nabla + b \cdot \nabla)v=-f \\
v \in W_0^{1,2}(\Omega),
\end{array}
\right.
\end{equation*}
(so  $u=v+g$), then
the right-hand side $f \equiv f_n:=-\nabla \cdot a \cdot \nabla g + b \cdot \nabla g$ is, in general,  not bounded in $ W^{-1,2}$ uniformly in $n$ (if it is uniformly bounded in $W^{-1,2}$, then the proof is straightforward).

\bigskip

One now arrives at the question: does a weak solution to \eqref{d0} constructed in Theorem \ref{thm2} (an ``approximation solution'') depend on the choice of the approximation procedure ($a_n$, $b_n$, $g_n$)? 
In the next theorem, which is essentially a consequence of Theorem \ref{thm1} and Gehring's Lemma, a ``generic constant'' can also depend on $\|g\|_{W^{2,2}(\Omega)}$.

\begin{theorem}
\label{thm3}
In the assumptions of Theorem \ref{thm2}, let also
$$g \in W^{2,2+\epsilon_1}_\loc \cap W^{1,\frac{1+\epsilon}{\epsilon}}_{\loc} \;\; \text{for some }0<\epsilon_1,\epsilon<1, \quad g_n \rightarrow g\;\;\text{ in }  W^{2,2+\epsilon_1}_\loc \cap W^{1,\frac{1+\epsilon}{\epsilon}}_{\loc},$$ 
and, in the assumption \eqref{nu} on $({\rm div\,}b)_+$, let $c_{\nu_+}=0$. 

There exists a generic  $p \in [1+\epsilon,2[$ such that if, in addition to \eqref{b_n}, $$|b| \in L^p_{\loc}, \quad b_n \rightarrow b\;\;\text{ in } [L^p_\loc]^d,$$ then the approximation solution to Dirichlet problem \eqref{d0} constructed in Theorem \ref{thm2} does not depend on the choice of $(a_n, b_n, g_n)$, and is in this sense unique.
\end{theorem}

Zhikov \cite{Zh} investigated approximation uniqueness for Dirichlet problem
\begin{equation*}
\left\{
\begin{array}{l}
(-\Delta + b \cdot \nabla)v=-f, \quad f \in W^{-1,2} \\
v \in W_0^{1,2}(\Omega)
\end{array}
\right.
\end{equation*}
with divergence-free $b$, singling out two classes of the approximation uniqueness results: 

1) $b=\nabla B$ with $B$ is of bounded mean oscillation on $\Omega$ or $\lim_{q \rightarrow \infty}q^{-1}\|B\|_{L^q(\Omega)}=0$ (the question what properties of $b$ ensure the existence of such $B$, satisfying some integrability assumptions, is non-trivial), 

2) $b \in [L^2(\Omega)]^d$ or $\lim_{\varepsilon \downarrow 0}\varepsilon\|b\|_{L^{2-\varepsilon}(\Omega)}=0$.

Theorem \ref{thm3} thus shows that one can step away from $p=2$ in the condition $|b| \in L^p(\Omega)$ by a fixed constant, provided that $b$ is multiplicatively form-bounded. One can now justifiably pose the question if one can remove the additional to $b \in \mathbf{MF}$ assumption ``$|b| \in L^p_{\loc}$ for a certain $1<p<2$'' completely.

\bigskip

\section{Smooth approximation of coefficients}

\label{approx_sect}

Let measurable
 $a \in (H_{\sigma,\xi})$, $b \in \mathbf{MF}$, assume that ${\rm div\,}b \in L^1_{\loc}$ satisfies \eqref{nu}, \eqref{nu2}, and boundary data  $g$ satisfies \eqref{g}. We discuss the question of constructing $a_n$, $b_n$, $g_n$ satisfying the assumptions \eqref{b_n}-\eqref{a_n} before Theorem \ref{thm2}.

It is not difficult to construct a bounded smooth approximation of matrix $a$ and $g$. 
The question of how to approximate $b$ by bounded smooth vector fields is more subtle since we need to control both the multiplicative form-bound and the form-bound of the (positive/negative part of the) divergence of the approximating vector fields $b_n$. We can put e.g.
$$
b_n:=\gamma_{\varepsilon_n} \ast b \quad \text{ for } \varepsilon_n \downarrow 0,
$$
where $\gamma_\varepsilon(y):=\varepsilon^{-d}\gamma(y/\varepsilon)$ is the Friedrichs mollifier, $\gamma(y):=ce^{-\frac{1}{|y|^2-1}}$ for $|y|<1$ and is zero otherwise, with $c$ adjusted to $\langle \gamma \rangle=1$, cf.\,\cite{KiS1,KiS3}. Indeed, the following is true for every $n \geq 1$:

1) $b_n \in L^\infty \cap C^\infty$. The second inclusion follows from $b \in [L^1_{\loc}]^d$ and the standard properties of Friedrichs mollifiers. To see the first inclusion, fix $x \in \mathbb R^d$ and estimate
\begin{align*}
|E_\varepsilon b(x)| & =|\langle b(\cdot)\sqrt{\gamma_\varepsilon(x - \cdot)},\sqrt{\gamma_\varepsilon(x - \cdot)}\rangle | \\
& \leq \delta \langle \big|\nabla \sqrt{\gamma_\varepsilon(x - \cdot)}\big|^2 \rangle^{\frac{1}{2}}   + c_\delta  \\
& (\text{we use $\big\langle \big|\nabla\sqrt{\gamma_\varepsilon(x - \cdot)}\big|^2\big\rangle=C^2\varepsilon^{-2}$}) \\
& \leq C\varepsilon^{-1}  + c_\delta
\end{align*}
for appropriate $C>0$ (clearly, independent of $x$). 

\smallskip

2) $b_n \in \mathbf{MF}_\delta$ with the same $c_\delta$ (thus, independent of $n$). 
Let $\varphi \in C_c^\infty$. First, let us note that for $\varphi_{m}:=\varphi+\frac{e^{-|x|^2}}{m}$ we have
\begin{align*}
\|\nabla\sqrt{\gamma_\varepsilon \ast |\varphi_m|^2}\|_2 & =\big\|\frac{\gamma_\varepsilon \ast(|\varphi_m||\nabla|\varphi_m|)}{\sqrt{\gamma_\varepsilon \ast|\varphi_m|^2}}\big\|_2 \\
& \leq \|\sqrt{\gamma_\varepsilon \ast|\nabla |\varphi_m||^2}\|_2=\|\gamma_\varepsilon \ast|\nabla |\varphi_m||^2\|_1^\frac{1}{2} \leq\|\nabla|\varphi_m|\|_2\leq \|\nabla \varphi_m\|_2
\end{align*}
(we need term $\frac{e^{-|x|^2}}{m}$ to make sure that $\gamma_\varepsilon \ast|\varphi_m|^2>0$ everywhere). Now, taking $m \rightarrow \infty$, we obtain
$$
\|\nabla\sqrt{\gamma_\varepsilon \ast |\varphi|^2}\|_2 \leq \|\nabla \varphi\|_2.
$$
By $b \in \mathbf{MF}_\delta$, we have for all $\varphi \in C_c^\infty$,
\begin{align*}
 \langle |\gamma_\varepsilon \ast b| \varphi,\varphi\rangle & \leq \langle |b|,\gamma_\varepsilon \ast |\varphi|^2\rangle \\
& \leq \delta \|\nabla \sqrt{\gamma_\varepsilon \ast |\varphi|^2}\|_2\|\sqrt{\gamma_\varepsilon \ast|\varphi|^2}\|_2 + c_\delta\|\sqrt{\gamma_\varepsilon \ast|\varphi|^2}\|_2 \\
& \leq \delta \|\nabla \varphi\|_2\|\varphi\|_2 + c_\delta\|\varphi\|_2^2,
\end{align*}
as needed.

\smallskip

3) We put ${\rm div\,}b_{n,\pm}:=E_{\varepsilon_n}({\rm div\,}b)_\pm \geq 0$, i.e.\,we mollify $({\rm div\,}b)_+:={\rm div\,}b \vee 0$ and $({\rm div\,}b)_-:=-({\rm div\,}b \wedge 0)$.
Then, clearly,
$${\rm div\,}b_n={\rm div\,}b_{n,+} - {\rm div\,}b_{n,-}.$$ 
The proof of the boundedness, smoothness and form-boundedness \eqref{nu}, \eqref{nu2} of ${\rm div\,}b_{n,+}$, ${\rm div\,}b_{n,-}$ follow the argument in 1), 2).

The convergence $b_n \rightarrow b$ in $[L^1_{\loc}]^d$, $({\rm div\,}b_n)_\pm \rightarrow {\rm div\,}b_\pm$ in $L^1_{\loc}$ follows from the properties of Friedrichs mollifiers.

Finally, the extra local summability/Sobolev regularity assumptions on $b$, $g$ in Theorem \ref{thm3} transfer to $b_n$, $g_n$ without any problems, by the properties of Friedrichs mollifiers.

\bigskip

\section{Caccioppoli's inequality}

\begin{proposition}
\label{c_prop}
Let $a$, $b$ satisfy the assumptions of Theorem \ref{thm1}. 
Let $u$ be a solution to equation \eqref{eq2} in a bounded Lipschitz domain $\Omega \subset \mathbb R^d$.
Set $v:=(u-c)_+$, $c \in \mathbb R$. Then, for all $x \in \Omega$ and all $0<r<R$, where $R$ is bounded from above by some $R_0 \leq 1$ such that $B_{R_0}(x) \subset\subset \Omega$,
we have
\begin{align}
\label{c_ineq} 
\langle |\nabla v|^2 \mathbf{1}_{B_{r}(x)}\rangle \leq K(R-r)^{-2} \langle v^2 \mathbf{1}_{B_{R}(x)}\rangle
\end{align}
for a generic constant $K$.
\end{proposition}

\begin{proof} 
\textit{Step 1 (a pre-Caccioppoli's inequality).} Without loss of generality, $x=0$. We fix $[0,1]$-valued smooth cut-off functions $\{\eta=\eta_{r_1,r_2}\}_{0<r_1<r_2<R}$ on $\mathbb R^d$ satisfying
$$
\eta=\left\{
\begin{array}{ll}1 & \text{ in } B_{r_1}, \\
0 & \text{ in } \mathbb R^d - \bar{B}_{r_2},
\end{array}
\right.
$$
and
\begin{equation}
\label{e1}
\frac{|\nabla \eta|^2}{\eta} \leq \frac{c}{(r_2-r_1)^2}\mathbf{1}_{B_{r_2}},
\end{equation}
\begin{equation}
\label{e2}
\sqrt{|\nabla \eta|} \leq \frac{c}{\sqrt{r_2-r_1}}\mathbf{1}_{B_{r_2}},
\end{equation}
\begin{equation}
\label{e3}
|\nabla \sqrt{|\nabla \eta|}| \leq \frac{c}{(r_2-r_1)^{\frac{3}{2}}}\mathbf{1}_{B_{r_2}}
\end{equation}
for some constant $c$. For instance, one can take, for $r_1 \leq |y| \leq r_2$, 
$$\eta(y):=1-\int_1^{1+\frac{|y|-r_1}{r_2-r_1}}\varphi(s)ds, \quad \text{where } \varphi(s):=Ce^{-\frac{1}{\frac{1}{4}-(s-\frac{3}{2})^2}}, \quad \supp \varphi=[1,2],$$
with constant $C$ adjusted to $\int_1^2 \varphi(s)ds=1$.

We multiply equation \eqref{eq2} by $v\eta$ and integrate to obtain
\begin{align*}
\langle a \cdot \nabla v,(\nabla v)\eta\rangle +  \langle a \cdot \nabla v,v\nabla \eta \rangle  + \langle b \cdot \nabla v,v\eta\rangle = 0,
\end{align*}
so
\begin{align*}
\sigma \langle |\nabla v|^2\eta\rangle +  \langle a \cdot \nabla v,v\nabla \eta \rangle  + \langle b \cdot \nabla v,v\eta\rangle \leq 0,
\end{align*}
Applying quadratic inequality in the second term, we obtain
\begin{align}
(\sigma-\epsilon)\langle |\nabla v|^2\eta\rangle & \leq \frac{1}{4\sigma^2\epsilon}\big\langle v^2\frac{|\nabla \eta|^2}{\eta} \big\rangle  - \langle b \cdot \nabla v,v\eta\rangle \quad (\epsilon>0) \notag \\
& \leq \frac{1}{4 \sigma^2 \epsilon}\big\langle v^2\frac{|\nabla \eta|^2}{\eta} \big\rangle  + \frac{1}{2}\langle bv,v\nabla \eta\rangle + \frac{1}{2}\langle {\rm div\,}b,v^2\eta\rangle \notag \\
&=:I_1 + I_2 + I_3. \label{i1} 
\end{align}
By \eqref{e1},
$$
I_1  \leq \frac{c}{4\sigma^2\epsilon (r_2-r_1)^2}\| v \mathbf{1}_{B_{r_2}}\|_2^2.
$$
Regarding $I_2$, we have
by $b \in \mathbf{MF}_\delta$,
\begin{align*}
2I_2 \leq \langle |b|v,v|\nabla \eta|\rangle & \leq \delta \|\nabla (v\sqrt{|\nabla \eta|})\|_2\|v\sqrt{|\nabla \eta|}\|_2 + c_\delta  \|v\sqrt{|\nabla \eta|}\|^2_2 \\
& \leq \delta \bigg(\|(\nabla v)\sqrt{|\nabla \eta|}\|_2 + \|v \nabla \sqrt{|\nabla \eta|}\|_2 \bigg) \, \|v\sqrt{|\nabla \eta|}\|_2 + c_\delta \|v\sqrt{|\nabla \eta|}\|^2_2.
\end{align*}
Hence, using \eqref{e2}, \eqref{e3}, we obtain 
\begin{align*}
2I_2  & \leq \delta c \bigg(\frac{1}{\sqrt{r_2-r_1}}\|(\nabla v)\mathbf{1}_{B_{r_2}}\|_2 + \frac{1}{(r_2-r_1)^{\frac{3}{2}}}\|v\mathbf{1}_{B_{r_2}}\|_2 \bigg)\, \frac{1}{\sqrt{r_2-r_1}}\|v\mathbf{1}_{B_{r_2}}\|_2 \\
&  + \frac{c_\delta c}{r_2-r_1}\|v\mathbf{1}_{B_{r_2}}\|_2^2.
\end{align*}
Thus, since $r_2-r_1<1$,
\begin{align*}
I_2  & \leq \frac{C_1}{r_2-r_1}\|(\nabla v)\mathbf{1}_{B_{r_2}}\|_2\| v\mathbf{1}_{B_{r_2}}\|_2\\
& + C_1 \biggl(1+\frac{1}{(r_2-r_1)^2} \bigg)\|v\mathbf{1}_{B_{r_2}}\|^2_2
\end{align*}
for appropriate constant $C_1$. 
Finally, recalling that ${\rm div\,}b={\rm div\,}b_+ - {\rm div\,}b_-$ for bounded smooth ${\rm div\,}b_\pm \geq 0$ that satisfy \eqref{nu}, \eqref{nu2}, we have by \eqref{nu}
\begin{align*}
I_3 & \leq \frac{1}{2}\langle {\rm div\,}b_+,v^2\eta\rangle \\
& \leq \frac{\nu_+}{2} \bigg( \langle |\nabla v|^2\eta \rangle + 4^{-1}\langle v^2 \frac{|\nabla \eta|^2}{\eta}\rangle \bigg) + \frac{c_{\nu_+}}{2} \langle v^2 \eta\rangle \\
& \leq \frac{\nu_+}{2} \langle |\nabla v|^2\eta \rangle + \frac{c_1}{(r_2-r_1)^2}\langle v^2\mathbf{1}_{B_{r_2}}\rangle, \quad c_1:=4^{-1}c\nu_+ + \frac{c_{\nu_+}}{2}.
\end{align*}

Substituting the above estimates on $I_1$, $I_2$ and $I_3$ in \eqref{i1}, selecting $\epsilon$ sufficiently small and using our assumption $\nu_+<2\sigma$, we obtain
\begin{align*}
\langle |\nabla v|^2\eta\rangle & \leq \frac{C_1}{r_2-r_1}\|(\nabla v)\mathbf{1}_{B_{r_2}}\|_2 \| v\mathbf{1}_{B_{r_2}}\|_2  \\
& + C_2 \biggl(1+\frac{1}{(r_2-r_1)^2} \bigg)\|v\mathbf{1}_{B_{r_2}}\|^2_2.
\end{align*}
Hence
\begin{align}
 \langle |\nabla v|^2\mathbf{1}_{B_{r_1}}\rangle & \leq \frac{C_1}{r_2-r_1}\|(\nabla v)\mathbf{1}_{B_{r_2}}\|_2\| v\mathbf{1}_{B_{R}}\|_2 \notag \\
& + C_2 \biggl(1+\frac{1}{(r_2-r_1)^2} \bigg)\|v\mathbf{1}_{B_{R}}\|^2_2. \label{proto_0}
\end{align}
We can divide \eqref{proto_0} by $\|v\mathbf{1}_{B_{R}}\|^2_2$:
\begin{align}
\frac{\|(\nabla v)\mathbf{1}_{B_{r_1}}\|_2^2}{\|v\mathbf{1}_{B_{R}}\|^2_2}  \leq \frac{C_1}{r_2-r_1}\frac{\|(\nabla v)\mathbf{1}_{B_{r_2}}\|_2}{\|v\mathbf{1}_{B_{R}}\|_2} + C_2 \biggl(1+\frac{1}{(r_2-r_1)^2} \bigg). \label{proto_cacc}
\end{align}
This is a pre-Cacciopolli inequality that we will now iterate.

\medskip

\textit{Step 2 (Caccioppoli's iterations).} Fix $r$ as in the formulation of the theorem (so $0<r<R$) and put in \eqref{proto_cacc}
$$
r_1:=R-\frac{R-r}{2^{n-1}}, \quad r_2:=R-\frac{R-r}{2^{n}}, \quad n=1,2,\dots
$$
so $r_2-r_1=\frac{R-r}{2^{n}}$. Then, denoting the LHS of \eqref{proto_cacc} by
$$
a^2_n:= \frac{ \|(\nabla v)\mathbf{1}_{B_{R-\frac{R-r}{2^{n-1}}}}\|_2^2}{\|v\mathbf{1}_{B_{R}}\|^2_2},
$$
the inequality \eqref{proto_cacc} can be written as
$$
 a_n^2 \notag  \leq C(R-r)^{-1}2^{n} a_{n+1} + C^2 (R-r)^{-2}2^{2n} + C^2
$$
for appropriate $C$ independent of $n$.
We multiply the latter by $(R-r)^2$ and divide by $C^2 2^{2n}$ $( \geq 1)$. Then, setting $$y_n:=\frac{(R-r) a_n}{C^2 2^{n}},$$ we obtain
\begin{equation}
\label{i_ineq}
y_n^2 \leq 1+(R-r)^2 + y_{n+1}, \quad n=1,2,\dots
\end{equation}
Now we can iterate \eqref{i_ineq}, estimating all $y_n$ via nested square roots
$1+(R-r)^2+\sqrt{1+(R-r)^2 + \sqrt{\dots}}$.
Or we can simply note that $\beta:=\sup_{n \geq 1} y_n$ satisfies $$\beta^2 \leq 1+(R-r)^2 +\beta.$$
(Note that $\beta<\infty$ since all $a_n$'s are bounded by a (non-generic) constant
$
\|(\nabla v)\mathbf{1}_{B_{R}}\|_2 / \|v\mathbf{1}_{B_{R}}\|_2<\infty.
$)
Hence
$$
\beta \leq  \frac{1+\sqrt{1+4(1+(R-r)^2)}}{2}
$$
which implies $\beta^2 \leq 3+2(R-r)^2$ and thus $y_n^2 \leq 3+2(R-r)^2$, $n=1,2,\dots$.So, taking $n=1$, 
we arrive at 
$$
\|(\nabla v)\mathbf{1}_{B_{r}}\|_2^2/\|v\mathbf{1}_{B_{R}}\|^2_2 \leq K (R-r)^{-2},
$$
for a generic constant $K$ (we used $R \leq 1$ to get rid of the constant term in the RHS). This is the claimed inequality.
\end{proof}

\bigskip

\section{Proof of Theorem \ref{thm1}}

\subsection{The $\sup$ bound}

The main result of this section is Proposition \ref{mp2}. It will follow from the next result.

\begin{proposition} 
\label{mp}
Fix $1<\theta<\frac{d}{d-2}$. 
There exists a generic constant $K$ such that for all $0 <R \leq R_0\leq 1$, where $B_{R_0}(x) \subset\subset \Omega$,
\begin{align}
\sup_{B_{\frac{R}{2}}(x)}u_+  \leq K \biggl(\frac{1}{|B_R(x)|} \langle u_+^{2\theta} \mathbf{1}_{B_R(x)}\rangle\biggr)^{\frac{1}{2\theta}}. \label{mp_f}
\end{align}
\end{proposition}

We prove Proposition \ref{mp}, armed with Proposition \ref{c_prop}, using  De Giorgi's method. Here we can follow Hara \cite{H}. For reader's convenience, we included the details.

\begin{proof}
Without loss of generality, $x=0$. 
Proposition \ref{c_prop} yields
$$
\|v\|_{W^{1,2}(B_{r})} \leq \tilde{K}(R-r)^{-1}\|v\|_{L^2(B_{R})}, \quad v:=(u-c)_+,\;\;c \in \mathbb R
$$
(we used $R-r \leq 1$).
By the Sobolev Embedding Theorem,
$$
\|v\|_{L^{\frac{2d}{d-2}}(B_{r})} \leq C(R-r)^{-1}\|v\|_{L^2(B_{R})}.
$$
So, by H\"{o}lder's inequality,
\begin{equation}
\label{ineq_h}
\|v\|_{L^{\frac{2d}{d-2}}(B_{r})} \leq C(R-r)^{-1}|B_{R}|^{\frac{\theta-1}{2\theta}}\|v\|_{L^{2\theta}(B_{R})}.
\end{equation}
Set $$R_m:=R\big(\frac{1}{2}+\frac{1}{2^{m+1}}\big), \quad B_m \equiv B_{R_m}, \quad m \geq 0,$$
so $B_R=B_0 \supset B_1 \supset \dots \supset B_{R/2}$.
Then, by \eqref{ineq_h},
$$
\|v\|_{L^{\frac{2d}{d-2}}(B_{m+1})}^2 \leq \hat{C} 2^{2m}|B_{m}|^{\frac{2}{d}+1-\frac{1}{\theta}}\|v\|_{L^{2\theta}(B_{m})}^2. 
$$
On the other hand,
$$
\frac{1}{|B_R|}  \langle v^{2\theta}\mathbf{1}_{B_{m+1}}\rangle \leq \biggl(\frac{1}{|B_R|}  \langle v^{\frac{2d}{d-2}}\mathbf{1}_{B_{m+1}}\rangle  \biggr)^{\theta\frac{d-2}{d}} \biggl(\frac{|B_{m+1} \cap \{v>0\}|}{|B_R|} \biggr)^{1-\theta\frac{d-2}{d}}.
$$
Applying the previous inequality in the first multiple in the RHS, we obtain
$$
\frac{1}{|B_R|}  \langle v^{2\theta}\mathbf{1}_{B_{m+1}}\rangle \leq \tilde{C}\frac{2^{2\theta m}}{|B_R|}\langle v^{2\theta}\mathbf{1}_{B_m} \rangle \biggl(\frac{|B_{m+1} \cap \{v>0\}|}{|B_R|} \biggr)^{1-\theta\frac{d-2}{d}}.
$$
Now, put $v_m:=(u-c_m)_+$ where $$c_m:=c(1-2^{-m}) \rightarrow c.$$ Then
$$
\frac{1}{c^{2\theta}|B_R|}  \langle v_{m+1}^{2\theta}\mathbf{1}_{B_{m+1}}\rangle \leq \tilde{C}\frac{2^{2\theta m}}{c^{2\theta}|B_R|}\langle v_{m+1}^{2\theta}\mathbf{1}_{B_m} \rangle \biggl(\frac{|B_{m+1} \cap \{u>c_{m+1}\}|}{|B_R|} \biggr)^{1-\theta\frac{d-2}{d}}.
$$
Hence, using
$$
\frac{|B_{m+1} \cap \{u>c_{m+1}\}|}{|B_R|} \leq \frac{(c_{m+1}-c_m)^{-2\theta}}{|B_R|} \langle v^{2\theta}_m \mathbf{1}_{B_{m+1}} \rangle,
$$
we obtain
$$
\frac{1}{c^{2\theta}|B_R|}  \langle v_{m+1}^{2\theta}\mathbf{1}_{B_{m+1}}\rangle \leq C 2^{2\theta m(2-\theta\frac{d-2}{d})} \biggl(\frac{1}{c^{2\theta}|B_R|}\langle v_m^{2\theta}\mathbf{1}_{B_{m}}\rangle\biggr)^{2-\theta\frac{d-2}{d}}.
$$
Denote
$
x_m:=\frac{1}{c^{2\theta}|B_R|}\langle v_m^{2\theta}\mathbf{1}_{B_{m}}\rangle
$
and fix $c$ by $$c^{2\theta}:=C^{\frac{1}{\alpha}}\gamma^{\frac{1}{\alpha^2}}\frac{1}{|B_R|}\langle v^{2\theta}\mathbf{1}_{B_{R}}\rangle \quad \text{ where } \alpha:=1-\theta\frac{d-2}{d},\;\; \gamma:=2^{2\theta (2-\theta\frac{d-2}{d})}.$$
Thus, we have
$$
x_{m+1} \leq C\gamma^m x_m^{1+\alpha}
$$
where, clearly, $x_0=C^{-\frac{1}{\alpha}}\gamma^{-\frac{1}{\alpha^2}}$. Hence, by a standard result \cite[Sect.7.2]{G}, $x_m \rightarrow 0$ as $m \rightarrow \infty$. It follows that
$$
\sup_{B_{R/2}}u_+ \leq c,
$$
which yields the claimed inequality.
\end{proof}

\begin{proposition} 
\label{mp2}
For every $0<p<\infty$ there exists a generic constant $K$ such that, for all $0 <R \leq R_0\leq 1$, where $B_{R_0}(x) \subset\subset \Omega$,
\begin{align*}
\sup_{B_{\frac{R}{2}}(x)}u  \leq K \biggl(\frac{1}{|B_R(x)|} \langle u_+^{p} \mathbf{1}_{B_R}(x)\rangle\biggr)^{\frac{1}{p}}. 
\end{align*}
\end{proposition}
\begin{proof}
We follow \cite{H}. If $p \geq 2\theta$ for some $\theta<\frac{d}{d-2}$, then the result follows by Proposition \ref{mp} and H\"{o}lder's inequality. If $0<p<2\theta$, then the proof goes as follows.
Proposition \ref{mp} yields: for all $0<r<R \leq R_0$ (without loss of generality, $x=0$),
$$
\sup_{B_r}u_+  \leq K \biggl(\frac{1}{(R-r)^d} \langle u_+^{2\theta} \mathbf{1}_{B_R}\rangle\biggr)^{\frac{1}{2\theta}}. 
$$
Hence
$$
\sup_{B_r}u_+ \leq K \biggl(\frac{1}{(R-r)^d} \langle u_+^{p} \mathbf{1}_{B_R}\rangle\biggr)^{\frac{1}{2\theta}}(\sup_{B_R}u_+)^{1-\frac{q}{2\theta}}.
$$
Now, applying Young's inequality, we arrive at
$$
\sup_{B_r}u_+ \leq \frac{1}{2}\sup_{B_R}u_+ + \tilde{K}\frac{1}{(R-r)^d} \langle u_+^{p} \mathbf{1}_{B_R}\rangle, \quad 0<r<R \leq 1.
$$
The result now follows upon applying \cite[Lemma 6.1]{G}.
\end{proof}

\bigskip

\subsection{The $\inf$ bound}

The main result of this section is Proposition \ref{inf_prop}.

\begin{proposition}
\label{pp_prop}
There exists generic constants $C$ and $q>0$ such that, if $u \geq c_0>0$ is a solution to \eqref{eq2} in $B_{2R}(x) \subset\subset \Omega$, then
\begin{equation*}
\biggl(\frac{1}{|B_R(x)|} \langle u^q \mathbf{1}_{B_R(x)}\rangle\biggr)\biggl(\frac{1}{|B_R(x)|} \langle u^{-q} \mathbf{1}_{B_R}(x)\rangle\biggr) \leq C^2.
\end{equation*}
\end{proposition}

The proof of Proposition \ref{pp_prop} consists of an iteration-type procedure similar to the one in the proof of Cacciopolli's inequality (Proposition \ref{c_prop}), and Moser's method, cf.\,\cite{H}.

\begin{proof}

\textit{Step 1.} We work over a ball $B_r(y) \subset B_{2R}(x)$, $r \leq 1$. Without loss of generality, $y=0$. Let $\zeta_{m}$ be $[0,1]$-valued smooth cut-off functions:
$$
\zeta_m(x)=\left\{
\begin{array}{ll}1 & |x| \leq r_{m}, \\
0 & |x| \geq r_{m+1},
\end{array}
\right. \qquad r_m:=r\left(1-\frac{1}{2^m} \right), \;\;m \geq 1,
$$
satisfying
$$
|\nabla \zeta_m| \leq C \frac{2^m}{r}, \qquad |\nabla|\nabla \zeta_m|| \leq C \frac{4^{m}}{r^2}.
$$ 
We multiply equation \eqref{eq2} by $u^{-1}\zeta_m^2$, obtaining, after integrating by parts, using $\sigma I \leq a \leq \xi I$ and applying quadratic inequality,
$$
\sigma \big\langle (\nabla w)^2 \zeta_m^2 \big\rangle \leq \varepsilon \xi \big\langle (\nabla w)^2 \zeta_m^2 \big\rangle + \frac{\xi}{\varepsilon} \langle (\nabla \zeta_m)^2 \rangle + \langle b \cdot \nabla w,\zeta_m^2\rangle.
$$
where $w:=\log u$. Hence, provided that $\varepsilon$ is fixed by $C_1:=\sigma-\varepsilon \xi>0$,
\begin{equation}
\label{i20}
C_1 \big\langle (\nabla w)^2 \zeta_m^2 \big\rangle \leq C_2 4^m r^{d-2} + \langle b \cdot \nabla w,\zeta_m^2\rangle.
\end{equation}
We need to estimate the last term. Integrating by parts, we have
$$
\langle b \cdot \nabla w,\zeta_m^2\rangle = - 2 \langle b (w-c),\zeta_m \nabla \zeta_m \rangle - \langle {\rm div\,}b,w\zeta_m^2 \rangle
$$
for any constant $c$ (we will chose its value later). Thus, 
\begin{equation}
\label{b__}
\langle b \cdot \nabla w,\zeta_m^2\rangle \leq 2 \big\langle |b||\nabla \zeta_m|^2\big\rangle^{\frac{1}{2}} \big\langle |b|(w-c)^2 \zeta_m^2 \big\rangle^{\frac{1}{2}} + \langle {\rm div\,}b_-,(w-c)\zeta_m^2 \rangle,
\end{equation}
where, recall ${\rm div\,}b={\rm div\,}b_+ - {\rm div\,}b_-$ for bounded smooth ${\rm div\,}b_\pm \geq 0$ that satisfy \eqref{nu}, \eqref{nu2}.

\smallskip

1)  We estimate the first multiple in the RHS of \eqref{b__}: by $b \in \mathbf{MF}$ (taking into account $r \leq 1$):
\begin{equation}
\label{i21}
\big\langle |b||\nabla \zeta_m|^2\big\rangle^{\frac{1}{2}} \leq C_3 2^{\frac{3}{2}m}r^{\frac{d-3}{2}}.
\end{equation}

2) 
The second multiple is estimated as
\begin{align*}
\big\langle |b|(w-c)^2 \zeta_m^2 \big\rangle^{\frac{1}{2}} & \leq \sqrt{\delta} \bigg( \|\zeta_m \nabla w\|_2 + \|(w-c)\nabla \zeta_m\|_2 \bigg)^{\frac{1}{2}}\|(w-c)\zeta_m\|^{\frac{1}{2}}_2 + \sqrt{c_\delta}\|(w-c)\zeta_m\|_2
\end{align*}
Therefore, setting $B_m:=B_{r_m}$, we have
\begin{align*}
\big\langle |b|(w-c)^2 \zeta_m^2 \big\rangle^{\frac{1}{2}} & \leq \sqrt{\delta} \bigg( \|\nabla w\|_{L^2(B_{m+1})} + 2^m r^{-1} \|w-c\|_{L^2(B_{m+1})} \bigg)^{\frac{1}{2}}\|w-c\|^{\frac{1}{2}}_{L^2(B_{m+1})} \\
& + \sqrt{c_\delta}\|w-c\|_{L^2(B_{m+1})}.
\end{align*}
Select $c$ to be the average $(w)_{B_{m+1}}:=|B_{m+1}|^{-1}\langle w\mathbf{1}_{B_{m+1}}\rangle$ of $w$ over $B_{m+1}$. Then, using the Poincar\'{e} inequality $\|w-c\|_{L^2(B_{m+1})} \leq C_0 r\|\nabla w\|_{L^2(B_{m+1})}$, we obtain
\begin{equation}
\label{i22}
\big\langle |b|(w-c)^2 \zeta_m^2 \big\rangle^{\frac{1}{2}} \leq C_4 2^{\frac{1}{2}m}r^{\frac{1}{2}}\|\nabla w\|_{L^2(B_{m+1})}.
\end{equation}

\smallskip

3) Finally, we estimate the last term in the RHS of \eqref{b__}:
\begin{equation}
\label{i220}
\langle {\rm div\,}b_-,(w-c)\zeta_m^2 \rangle \leq  \langle {\rm div\,}b_-,(w-c)^2\zeta_m^2  \rangle^{\frac{1}{2}}\langle {\rm div\,}b_-,\zeta_m^2  \rangle^{\frac{1}{2}},
\end{equation}
where, by the form-boundedness assumption \eqref{nu2},
\begin{align}
 \langle {\rm div\,}b_-,& (w-c)^2\zeta_m^2  \rangle \notag \\
& \leq \nu_-(1+\varepsilon_1)\|(\nabla w)\zeta_m\|_2^2 + \nu_- (1+\varepsilon_1^{-1})\|(w-c)\nabla \zeta_m\|_2^2 + c_{\nu_-} \|(w-c)\zeta_m\|_2^2  \label{i22_} \\
& \leq \nu_-(1+\varepsilon_1)\|(\nabla w)\zeta_m\|_2^2 + \nu_- (1+\varepsilon_1^{-1})4^{m}r^{-2}\|w-c\|_{L^2(B_{m+1})}^2 + c_{\nu_-} \|w-c\|_{L^2(B_{m+1})}^2 \notag
\end{align}
and
\begin{align}
\langle {\rm div\,}b_-,\zeta_m^2  \rangle & \leq \nu_-\|\nabla \zeta_m\|_2^2 + c_{\nu_-} \|\zeta_m\|_2^2 \notag \\
& \leq C_5 4^mr^{d-2}. \label{i22__}
\end{align}
Hence, applying \eqref{i22_}, \eqref{i22__} in \eqref{i220}, using the quadratic inequality and the Poincar\'{e} inequality as above, we obtain
\begin{equation}
\label{div_}
\langle ({\rm div\,}b)_-,(w-c)\zeta_m^2 \rangle \leq  C_6 \varepsilon_2 \|(\nabla w)\zeta_m\|_2 + \frac{C_7}{4\varepsilon_2}  4^mr^{d-2} + C_7 4^m r^{\frac{d}{2}-1}\|\nabla w\|_{L^2(B_{m+1})}.
\end{equation}

We now apply \eqref{i22_}, \eqref{i22__} and \eqref{div_} (with $\varepsilon_2$ chosen sufficiently small) in \eqref{i220}, arriving at
$$
\big\langle (\nabla w)^2 \zeta_m^2 \big\rangle \leq C^2 4^{m} r^{d-2} + C 4^{m} r^{\frac{d}{2}-1}\|\nabla w\|_{L^2(B_{m+1})}
$$
for appropriate constant $C$.
Therefore,
\begin{equation}
\label{i23}
\|\nabla w\|_{L^2(B_{m})}^2 \leq C^2 4^{m} r^{d-2} + C 4^{m} r^{\frac{d}{2}-1}\|\nabla w\|_{L^2(B_{m+1})}
\end{equation}
for all $m=1,2,\dots$

\medskip

\textit{Step 2.~}We are going to iterate inequality \eqref{i23}. Put
$$
x_m:=\frac{\|\nabla w\|_{L^2(B_m)}}{C 2^m r^{\frac{d}{2}-1}}
$$
so \eqref{i23} becomes
$$
x_m^2 \leq 1 + 2^{m+1} x_{m+1}.
$$
We may assume without loss of generality that all $x_m \geq 1$ (if $x_{m_0} \leq 1$ for some $m_0$, then we are already done). Thus,
$$
x_m^2 \leq 2^{m+2}x_{m+1}.
$$
On the other hand, all $x_m$ ($m=1,2,\dots$) are bounded by a non-generic but independent of $m$ constant $\frac{\|\nabla w\|_{L^2(B_r)}}{2C  r^{\frac{d}{2}-1}}$.
Hence we can iterate the previous inequality:
$$
x_m^2 \leq 2^{m+2}x_{m+1} \leq 2^{m+2}2^{\frac{m+3}{2}}x_{m+2}^{\frac{1}{2}} \leq 2^{m+2}2^{\frac{m+3}{2}} \dots 2^{\frac{m+n}{2^n}} x_{m+1+n}^{\frac{1}{2^n}} \leq 2^{\sum_{n=1}^\infty \frac{m+1+n}{2^n}}=:c(m).
$$
In particular, $x_1^2 \leq c(1)$, which yields
\begin{equation}
\label{ineq_w}
\|\nabla w\|_{L^2(B_{r/2})} \leq Kr^{\frac{d}{2}-1}
\end{equation}
for a generic constant $K$. 

\medskip

\textit{Step 3.}~Inequality \eqref{ineq_w} is the point of departure for Moser's method.
Namely, applying Poincar\'{e}'s and H\"{o}lder's inequalities, we obtain
\begin{equation}
\label{bmo}
\frac{1}{r^d}\langle |w-(w)_{r/2}|\mathbf{1}_{B_{r/2}}\rangle \leq \tilde{K},
\end{equation}
where $(w)_{r/2}$ is the average of $w$ over $B_{r/2}$.

The centre $y$ of the ball  in \eqref{bmo} was chosen arbitrarily, and the constant $\tilde{K}$ does not depend on this choice. Thus, by \eqref{bmo}, $w \in {\rm BMO}(B_{2R}(x))$ (in what follows, for brevity, $x=0$). Now, by the John-Nirenberg inequality:
$$
\langle e^{q|w-(w)_{R}|}\mathbf{1}_{B_R} \rangle \leq C R^d
$$
for some generic $q>0$. So, we have
\begin{align*}
\langle u^q \mathbf{1}_{B_R}\rangle  \langle u^{-q} \mathbf{1}_{B_R} \rangle & =\langle e^{qw} \mathbf{1}_{B_R}\rangle \langle e^{-qw} \mathbf{1}_{B_R}\rangle \\
& = \langle e^{q(w-(w)_R)} \mathbf{1}_{B_R} \rangle \langle e^{-q(w-(w)_R)} \mathbf{1}_{B_R}\rangle\leq C^2R^{2d},
\end{align*}
as needed.
\end{proof}

\begin{proposition}[Moser]
\label{inf_prop}
There exists generic constants $C_0$ and $q>0$ such that, if $u \geq c_0>0$ is a solution to \eqref{eq2} in $B_{2R}(x) \subset \Omega$, then
$$
\biggl(\frac{1}{|B_R(x)|} \langle u^{q} \mathbf{1}_{B_R(x)}\rangle\biggr)^{\frac{1}{q}} \leq C_0 \inf_{B_{R/2}(x)}u. 
$$
\end{proposition}
\begin{proof}
Let $x=0$.
Multiplying equation \eqref{eq2} by $u^{-p}$ ($p>1$), we obtain that $u^{-p+1}$ is a sub-solution of \eqref{eq2}:
$$
-\nabla \cdot a \cdot \nabla u^{-p+1} + b \cdot \nabla u^{-p+1} \leq 0.
$$ 
We can repeat the proofs of Proposition \ref{c_prop} and of Proposition \ref{mp} for positive sub-solutions of \eqref{eq2} essentially word by word.
Fix $p$ by $p-1=\frac{q}{2\theta}$ for any $1<\theta<\frac{d}{d-2}$. Then, by Proposition \ref{mp},
$$
\sup_{B_{R/2}}u^{-1} \leq K^{\frac{2\theta}{q}} \biggl(\frac{1}{|B_R|} \langle u^{-q} \mathbf{1}_{B_R}\rangle\biggr)^{\frac{1}{q}}.
$$
Hence
\begin{align*}
\inf_{B_{R/2}}u & \geq K^{-\frac{2\theta}{q}} \biggl(\frac{1}{|B_R|} \langle u^{-q} \mathbf{1}_{B_R}\rangle\biggr)^{-\frac{1}{q}} \\
& \text{(we are applying Proposition \ref{pp_prop})} \\
& \geq C^{-\frac{2}{q}}K^{-\frac{2\theta}{q}}  \biggl(\frac{1}{|B_R|} \langle u^{q} \mathbf{1}_{B_R}\rangle\biggr)^{\frac{1}{q}},
\end{align*}
as needed.
\end{proof}

\medskip

We are in position to complete the proof of Theorem \ref{thm1}. 
Propositions \ref{mp2} and \ref{inf_prop} yield Harnack's inequality for $u \geq c_0>0$ in $B_{2R}(x)$. A simple limiting argument allows to extend it to $u \geq 0$ in $B_{2R}(x)$. The H\"{o}lder continuity of  $u$ now follows using a standard argument. The gradient estimate is a standard consequence of Cacciopolli's inequality (Proposition \ref{c_prop}), see e.g.\,\cite{H}. \hfill \qed

\begin{remark}
\label{bb_rem}
The iteration procedure of Propositions \ref{c_prop} and \ref{pp_prop} extends to vector fields
\begin{equation*}
b=b_1+b_2, \quad b_1 \in \mathbf{MF}, \quad b_2 \in \mathbf{BMO}^{-1}
\end{equation*}
where ${\rm div\,}b_1$ satisfies the assumptions of Theorem \ref{thm1}, $b_2=\nabla F$ for skew-symmetric $F$ with entries in ${\rm BMO}$. Namely, repeating the proof of Proposition \ref{c_prop} for such $b$ (cf.\,the proof of \eqref{inv_holder} below), one obtains an extra term $\langle b_2 \cdot \nabla v,v\eta\rangle$, which one estimates as in \cite[Lemma 8]{H}:
$$
|\langle b_2 \cdot \nabla v,v\eta\rangle| \leq \varepsilon \langle |\nabla u|\eta\rangle+ \frac{1}{4\varepsilon}\langle (F-(F)_{B_R})^2 v^2 \frac{|\nabla \eta|^2}{\eta} \rangle, \quad \varepsilon>0,
$$
and so the pre-Caccioppoli's inequality now takes form
\begin{align*}
\frac{\|(\nabla v)\mathbf{1}_{B_{r_1}}\|_2^2}{\|v\mathbf{1}_{B_{R}}\|^2_2}  \leq \frac{C_1}{r_2-r_1}\frac{\|(\nabla v)\mathbf{1}_{B_{r_2}}\|_2}{\|v\mathbf{1}_{B_{R}}\|_2} + C_2 \biggl(1+\frac{1}{(r_2-r_1)^2} \bigg)\bigg(1+\frac{\|(F-(F)_{B_R})v\mathbf{1}_{B_{R}}\|_2^2}{\|v\mathbf{1}_{B_{R}}\|^2_2}\bigg). 
\end{align*}
One can now iterate this inequality in the same was as it was done in the proof of Proposition \ref{c_prop}, arriving, instead of \eqref{c_ineq}, at a Caccioppoli-type inequality as in \cite{H} (cf.\,(b) in the introduction):
$$
\langle |\nabla v|^2 \mathbf{1}_{B_{r}}\rangle \leq K(R-r)^{-2} \langle [1+(F-(F)_{B_R})^2] v^2 \mathbf{1}_{B_{R}}\rangle, \quad B_R=B_R(x).
$$
Having the last inequality at hand, one then runs De Giorgi's method as in \cite{H}. The proof of Proposition \ref{pp_prop} is modified similarly; then we refer again to \cite{H}. This allows to extend Theorem \ref{thm1} to vector fields $b=b_1+b_2$ as above.

\end{remark}

\bigskip

\section{Proof of Theorem \ref{thm2}}

Given $\Omega \subset \mathbb R^d$, put $\|f\|_{p,\Omega}:=(\int_{\Omega} |f|dx)^{\frac{1}{p}}$.

\medskip

It is convenient to put Dirichlet problem \eqref{d0} in an equivalent form (at the level of weak solutions)
\begin{equation}
\label{d3}
\left\{
\begin{array}{l}
(-\nabla \cdot a \cdot \nabla + b \cdot \nabla)v=-f \\
v \in W_0^{1,2}(\Omega),
\end{array}
\right.
\end{equation}
where $f:=-\nabla \cdot a \cdot \nabla g + b \cdot \nabla$. Then the sought $u$ is given by $u=v+g$. 

Considering \eqref{d3} for $a=a_n$, $b=b_n$, $g=g_n$, $f=f_n$ and, accordingly, $v=v_n$, we multiply the equation by $v_n$ and integrate to obtain
$$
\sigma\|\nabla v_n\|_{2,\Omega}^2 \leq \frac{1}{2}\langle {\rm div\,}b_n,v_n^2\rangle - \langle f_n,v_n\rangle,
$$
so, in particular,
$$
c\sigma \epsilon \|v_n\|_{2,\Omega}^2 + \sigma (1-\epsilon)\|\nabla v_n\|_{2,\Omega}^2  \leq \frac{1}{2}\langle {\rm div\,}b_{n,+},v_n^2\rangle - \langle f_n,v_n\rangle
$$
for a small $\epsilon>0$.
In the RHS of the last inequality, using the fact that ${\rm div\,}b_{n,+}$ satisfies \eqref{nu}, we estimate
\begin{align*}
\frac{1}{2}\langle {\rm div\,}b_{n,+},v_n^2\rangle \leq \frac{\nu_+}{2}\|\nabla v_n\|_2^2 + \frac{c_{\nu_+}}{2}\|v_n\|_2^2
\end{align*}
and, applying quadratic inequality twice, we have
\begin{align*}
\big|\langle f_n,v_n\rangle\big| & \leq \sigma \|\nabla g_n\|_{2,\Omega}\|\nabla v_n\|_{2,\Omega} + \alpha \langle |b_n|,v_n^2 \rangle + \frac{1}{4\alpha}\langle |b_n|,|\nabla g_n|^2\rangle \qquad (\alpha>0)\\
& (\text{we are using $b_n \in \mathbf{MF}_\delta$ for some $\delta<\infty$, $c_\delta$ independent of $n$}) \\
& \leq \sigma \beta \|\nabla v_n\|_{2,\Omega}^2+ \frac{\sigma}{4\beta} \|\nabla g_n\|_{2,\Omega}^2 \qquad (\beta>0) \\
& + \alpha (\delta\|\nabla v_n\|_{2,\Omega}\|v_n\|_{2,\Omega}+c_\delta\|v_n\|_{2,\Omega}^2) \\
& + \frac{1}{4 \alpha }(\delta\|\nabla |\nabla g_n|\|_{2,\Omega}\|\nabla g_n\|_{2,\Omega}+c_\delta\|g_n\|_{2,\Omega}^2).
\end{align*}
Now, selecting in the previous three inequalities $\alpha$, $\beta$ sufficiently small and using $\nu_+<2\sigma$, we obtain
\begin{align*}
\|v_n\|_{W^{1,2}(\Omega)}^2 & \leq C\|g\|_{W^{2,2}(\Omega)} + C_1\|v_n\|_{2,\Omega} \\
& \leq C\|g\|_{W^{2,2}(\Omega)} + C_1|\Omega|^{\frac{1}{2}}\|v_n\|_{\infty,\Omega}
\end{align*}
for some $C$, $C_1<\infty$. Hence, taking into account that, by the maximum principle, $\|v_n\|_{L^\infty(\Omega)} \leq 2\|g_n\|_\infty (\leq 2\|g\|_\infty<\infty)$, we have
\begin{equation}
\label{nabla_v}
\|v_n\|_{W^{1,2}(\Omega)}^2 \leq C\|g\|_{W^{2,2}(\Omega)}  + 2C_1|\Omega|^{\frac{1}{2}}\|g\|_\infty.
\end{equation}
This allows to conclude that there exists a subsequence $\{v_n\}$ and a function $v \in W^{1,2}_0(\Omega)$ such that
\begin{equation}
\label{v_conv}
v_n \rightarrow v \quad \text{ weakly in } W_0^{1,2}(\Omega), \text{ strongly in $L_{\loc}^2(\Omega)$.}
\end{equation}

Let us show that thus constructed $v$ is a weak solution to \eqref{d3}:

1) For a given $\varphi \in C_c^\infty(\Omega)$, we can write
\begin{align*}
\langle a_n \cdot \nabla v_n,\nabla \varphi\rangle + \langle b_n \cdot \nabla v_n,\varphi\rangle & = \langle a_n \cdot \nabla v_n,\nabla \varphi\rangle - \langle b_n  v_n, \nabla \varphi\rangle - \langle ({\rm div\,}b_n)v_n,\varphi\rangle \\
& = \langle a \cdot \nabla v,\nabla \varphi\rangle - \langle b v,\nabla \varphi\rangle - \langle {\rm div\,b} v,\varphi\rangle\\
& +
\langle (a_n-a) \cdot \nabla v_n,\nabla \varphi\rangle + \langle a\cdot (\nabla v_n-\nabla v),\nabla \varphi\rangle \\
& - \langle (b_n-b) v_n,\nabla \varphi\rangle - \langle b(v_n-v),\nabla \varphi\rangle \\
& - \langle ({\rm div\,}b_n-{\rm div\,}b)v_n,\nabla \varphi \rangle -  \langle ({\rm div\,}b)(v_n-v),\varphi \rangle.
\end{align*}

The $-6$th term in the RHS, i.e.\,$\langle (a_n-a) \cdot \nabla v_n,\nabla \varphi\rangle$, tends to $0$: we use $\|\nabla v_n\|_{L^p(\supp \varphi)}<\infty$ for some $p>2$ (Theorem \ref{thm1}) and convergence $a_n \rightarrow a$ in $L^{p'}(\supp \varphi)$ (use \eqref{a_n} and $a_n \in (H_{\sigma,\xi})$).

The $-5$th term $\langle a\cdot (\nabla v_n-\nabla v),\nabla \varphi\rangle$, tends to $0$ since $v_n \rightarrow v$ weakly in $W_0^{1,2}(\Omega)$. 

The $-4$th term $\langle (b_n-b) v_n,\nabla \varphi\rangle$ tends to $0$ due to \eqref{b_n} and since $v_n$ are uniformly bounded on $\Omega$ ($\|v_n\|_{L^\infty(\Omega)} \leq 2\|g_n\|_\infty$ by the maximum principle, where, by our choice of $g_n$, $\|g_n\|_\infty \leq \|g\|_\infty < \infty$). 

The $-3$rd term $\langle b(v_n-v),\nabla \varphi\rangle$ goes to $0$ by $|b| \in L^1_{\loc}$ and the Dominated Convergence Theorem, using the uniform boundedness of $v_n$ on $\Omega$ and a.e.\,convergence $v_n \rightarrow v$, which follows from \eqref{v_conv} (possibly after passing to a subsequence using a diagonal argument).

The $-2$nd term $\langle ({\rm div\,}b_n-{\rm div\,}b)v_n,\nabla \varphi \rangle$ goes to $0$ by convergence \eqref{div_n2} and since $v_n$ are uniformly bounded on $\Omega$.

The $-1$st term $\langle ({\rm div\,}b)(v_n-v),\varphi \rangle$ goes to $0$ by ${\rm div\,}b \in L^1_{\loc}$ and, again, the Dominated Convergence Theorem, using uniform boundedness of $v_n$ and a.e.\,convergence $v_n \rightarrow v$.

\smallskip

2) Next, in view of our assumptions on $g$ and $g_n$,
\begin{align*}
\langle f_n,\varphi\rangle & = \langle a_n \cdot \nabla g_n,\nabla \varphi\rangle + \langle b_ng_n,\nabla \varphi\rangle \\
& \rightarrow \langle f,\varphi\rangle
\end{align*}
using the same argument as in 1), taking into account that, by \eqref{g_n} and the Rellich-Kondrashov Theorem, we may assume that $g_n \rightarrow g$ a.e.\,on $\Omega$ (of course, possibly after passing to a subsequence of $\{g_n\}$)

Combining 1) and 2), we obtain that $v$ is a weak solution to \eqref{d3}. Moreover, since $v_n$ are (uniformly in $n$) bounded on $\Omega$, so is $v$. 

Now, we have $u_n=v_n+g_n$, so, in view of our conditions on $g_n$ and $g$,
$$
u_n \rightarrow u:=v+g \quad \text{ weakly in } W^{1,2}_0(\Omega),
$$
and so $u_n \rightarrow u$ strongly in $L^2_{\loc}(\Omega)$, possibly after passing to a subsequence.
Further, since  $u_n$ are bounded on $\Omega$ by the maximum principle, so is $u$. The last statement of the theorem now follows from Theorem \ref{thm1}. \hfill \qed

\bigskip

\section{Proof of Theorem \ref{thm3}}

\textbf{1.~}To establish uniqueness of the approximation solution, it suffices to show that solutions $\{v_n\}$ to
\begin{equation}
\label{d4}
\left\{
\begin{array}{l}
(-\nabla \cdot a_n \cdot \nabla + b_n \cdot \nabla)v_n=-f_n \\
v_n=0 \text{ on } \partial\Omega,
\end{array}
\right.
\end{equation}
where $f_n:=-\nabla \cdot a_n \cdot \nabla g_n + b_n \cdot \nabla g_n$,
constitute a Cauchy sequence in $L^2(\Omega)$. (Then, clearly, solutions $u_n=v_n+g_n$ to \eqref{d2} constitute a Cauchy sequence in $L^2(\Omega)$.)
In fact, subtracting the equations for $v_n$, $v_m$ and setting $h:=v_n-v_m$, we obtain
$$
-\nabla \cdot a_n \cdot \nabla h  + b_n \cdot \nabla h  - \nabla \cdot (a_n-a_m) \cdot \nabla v_m + (b_n-b_m) \cdot \nabla v_m=-f_n+f_m, 
$$
Then, multiplying the previous identity by $h$ and integrating, we obtain
\begin{align*}
\sigma\|\nabla h\|^2_2 - \frac{1}{2}\langle {\rm div\,}b_n,h^2\rangle & \leq |\langle (a_n-a_m) \cdot \nabla v_m,\nabla h\rangle| + |\langle (b_n-b_m) \cdot \nabla v_m,h\rangle| \notag  \\
& + |\langle f_n-f_m,h\rangle|.
\end{align*}
Hence, using $\frac{1}{2}\langle ({\rm div\,}b_{n,+}),h^2\rangle \leq \frac{1}{2}\nu_+\|\nabla h\|^2_2$ (by the assumption of the theorem), we have
\begin{align}
(\sigma-\nu_+)\|\nabla h\|_2   & \leq |\langle (a_n-a_m) \cdot \nabla v_m,\nabla h\rangle| + |\langle (b_n-b_m) \cdot \nabla v_m,h\rangle| \notag  \\
& + |\langle f_n-f_m,h\rangle|.
\label{a_b}
\end{align}
Recall that, by our assumption, $\sigma-\frac{\nu_+}{2}>0$.
Thus, our goal is to show that all terms in the RHS of \eqref{a_b} tend to $0$ as $n$, $m \rightarrow \infty$; this would imply that $\{v_m\}$ is indeed a Cauchy sequence in $L^2(\Omega)$.

\smallskip

1) Let us get rid of the last term in the RHS of \eqref{a_b}:
\begin{align*}
\langle f_n-f_m,h\rangle=\langle (a_n-a_m)\rangle & = \langle (a_n-a_m) \cdot \nabla g_n,\nabla h\rangle \\
& + \langle a_m \cdot \nabla (g_n-g_m),\nabla h\rangle \\
& + \langle (b_n-b_m)\cdot \nabla g_n,h\rangle \\
& + \langle b_m \cdot \nabla (g_n-g_m),h\rangle.
\end{align*}
All four terms in the RHS tends to $0$ as $n,m \rightarrow \infty$. This follows, upon applying H\"{o}lder's inequality, from the uniform boundedness of $|\nabla h|$ in $L^2(\Omega)$ (cf.\,\eqref{nabla_v} in  the proof of Theorem \ref{thm2}) and convergence $\nabla g_n-\nabla g_m \rightarrow 0$ in $[L^{\frac{1+\epsilon}{\epsilon}}_\loc]^d$, $b_n-b_m \rightarrow 0$ in $[L^p_{\loc}]^d$ ($p \geq 1+\epsilon$) as $n,m \rightarrow \infty$.

\smallskip

2) We now treat the first two terms in the RHS of \eqref{a_b}. 
\begin{align*}
|\langle (b_n-b_m) \cdot \nabla v_m,h\rangle| & \leq \|b_n-b_m\|_{p,\Omega}\|\nabla v_m\|_{p',\Omega}\|h\|_{\infty,\Omega} \\
& \leq \|b_n-b_m\|_{p,\Omega}\|\nabla v_m\|_{p',\Omega}\,2\|g\|_{\infty}.
\end{align*}
If we can prove a uniform in $n$ bound
\begin{equation}
\label{grad2}
\|\nabla v_m\|_{p',\Omega} \leq C \quad \text{ for some } p'>2,
\end{equation}
it would imply that $\langle (b_n-b_m) \cdot \nabla v_m,h\rangle \rightarrow 0$ as $n, m \rightarrow \infty$, since $b_n-b_m \rightarrow 0$ in $[L^p_{\loc}]^d$ by the assumption of the theorem, with $p\,(=\frac{p'}{p'-1})<2$.

The estimate \eqref{grad2} is also what is needed 
to prove $\langle (a_n-a_m) \cdot \nabla v_m,\nabla h\rangle \rightarrow 0$, since $a_n -a_m \rightarrow 0$ in $L^q(\Omega)$ for any $q<\infty$.

Thus, the proof of Theorem \ref{thm3} will be completed once we prove \eqref{grad2}.

\medskip

\textbf{2.~}Proof of \eqref{grad2}. Write for brevity $v=v_m$, $a=a_m$, $b=b_m$ (the constants below are independent of $m$). We extend $v$ to $\mathbb R^d$ by zero. It suffices to establish
\begin{equation}
\label{inv_holder}
(R/4)^{-d}\langle |\nabla v|^2\mathbf{1}_{B_{\frac{R}{4}}(x)}\rangle \leq C \biggl[\bigg(R^{-d}\langle |\nabla v|^\frac{2}{\theta}\mathbf{1}_{B_R}(x) \rangle\bigg)^{\theta} + R^{-d}\langle k^2 \rangle\biggr]
\end{equation}
for generic constants $\theta>1$, $C$, for all $R \leq R_0$, $x \in \Omega$, for some function $k \in L_{\loc}^{2+\epsilon}$, $\epsilon>0$. Then Gehring's Lemma will yield \eqref{grad2}.

Let us prove \eqref{inv_holder}. If $B_{\frac{R}{2}}(x) \subset \Omega$, then we put $w:=v-(v)_{B_{R}(x)},$ otherwise $w:=v$. Without loss of generality, $x=0$.
As in the proof of Proposition \ref{c_prop}, we 
fix $[0,1]$-valued smooth cut-off functions $\{\eta=\eta_{r_1,r_2}\}_{0<r_1<r_2<R}$ on $\mathbb R^d$ such that
$$
\eta=\left\{
\begin{array}{ll}1 & \text{ in } B_{r_1}, \\
0 & \text{ in } \mathbb R^d - \bar{B}_{r_2},
\end{array}
\right.
$$
satisfying \eqref{e1}-\eqref{e3}.
We multiply equation $-\nabla \cdot a \cdot \nabla w + b \cdot \nabla w=-f$ by $w\eta$, integrate over $\mathbb R^d$, and argue as in the proof of \eqref{proto_0} to obtain
\begin{align}
 \sigma\langle |\nabla w|^2 \eta\rangle & \leq \frac{C_1}{r_2-r_1}\|(\nabla w)\mathbf{1}_{B_{r_2}}\|_2\| w\mathbf{1}_{B_{r_2}}\|_2 \notag \\
& + C_2 \biggl(1+\frac{1}{(r_2-r_1)^2} \bigg)\|w\mathbf{1}_{B_{\frac{R}{2}}}\|^2_2 \notag \\
&  + |\langle f,w\eta\rangle|. \label{proto_2}
\end{align}
In comparison with \eqref{proto_0}, we now have an extra term $|\langle f,w\eta\rangle|$. We deal with it as follows:
\begin{align*}
|\langle f,w\eta\rangle| & = |\langle a \cdot \nabla g,(\nabla w)\eta + w\nabla \eta\rangle + \langle b \cdot \nabla g,w\eta\rangle| \\
& \leq \alpha \langle|\nabla w|^2 \eta \rangle + \alpha \langle w^2 |\nabla \eta|\rangle + \frac{\xi^2}{4\alpha}\langle|\nabla g|^2\eta\rangle + \gamma\langle |b|w^2\eta\rangle + \frac{1}{4\gamma} \langle |b| |\nabla g|^2 \eta\rangle \quad (\alpha,\gamma>0).
\end{align*}
Now, applying $b \in \mathbf{MF}$ and
substituting the result in \eqref{proto_2}, we obtain
\begin{align*}
(\sigma-\alpha-\gamma) \langle |\nabla w|^2 \eta\rangle & \leq \frac{C_1}{r_2-r_1}\|(\nabla w)\mathbf{1}_{B_{r_2}}\|_2\| w\mathbf{1}_{B_{r_2}}\|_2 \notag \\
& + C_2 \biggl(1+\frac{1}{(r_2-r_1)^2} \bigg)\|w\mathbf{1}_{B_R}\|^2_2   + C_3(\alpha,\gamma) \langle w^2 |\nabla \eta|\rangle + C_4(\alpha,\gamma)\|g\|_{W^{2,2}(B_R)}.
\end{align*}
Hence, fixing $\alpha$ and $\gamma$ sufficiently small so that  $\sigma-\alpha-\gamma>0$, we obtain
\begin{align*}
\langle |\nabla w|^2 \mathbf{1}_{r_1}\rangle & \leq \frac{C'_1}{r_2-r_1}\|(\nabla w)\mathbf{1}_{B_{r_2}}\|_2\| w\mathbf{1}_{B_{r_2}}\|_2 \notag \\
& + C'_2 \biggl(1+\frac{1}{(r_2-r_1)^2} \bigg)\|w\mathbf{1}_{B_{R}}\|^2_2  + C_4'(\alpha,\gamma)\|g\|_{W^{2,2}(B_{R})}.
\end{align*}
We now iterate this inequality in the same way as in the proof of Proposition \ref{c_prop}, selecting 
$$
r_1:=R-\frac{R}{2^{n-1}}, \quad r_2:=R-\frac{R}{2^{n}}, \quad n=1,2,\dots,
$$
arriving, upon taking $n \rightarrow \infty$, to 
$$
\|(\nabla w)\mathbf{1}_{B_{\frac{R}{4}}}\|^2_2 \leq C\biggl[\frac{1}{R^2}\|w\mathbf{1}_{B_{R}}\|^2_2 + \|g\|_{W^{2,2}(B_{R})}\biggr].
$$
By the Sobolev-Poincar\'{e} inequality (or by the Sobolev inequality, if $B_{\frac{R}{2}}(x) \not\subset \Omega$),
we have
$$
\|(\nabla w)\mathbf{1}_{B_{\frac{R}{4}}}\|^2_2 \leq C\biggl[\frac{1}{R^2}\|(\nabla w)\mathbf{1}_{B_{R}}\|^2_{\frac{2d}{d+2}} + \|g\|_{W^{2,2}(B_R)}\biggr], 
$$
so
$$
R^{-d}\langle |\nabla w|^2\mathbf{1}_{B_{\frac{R}{4}}}\rangle \leq C\biggl[\bigl(R^{-d}\langle|\nabla w|^{\frac{2}{\theta}}\mathbf{1}_{B_{R}}\rangle\bigr)^{\theta} + R^{-d}\|g\|_{W^{2,2}(B_{R})}\biggr],  \quad \theta=\frac{d+2}{d}.
$$
Now Gehring's Lemma yields \eqref{grad2} and thus ends the proof. \hfill \qed

\bigskip

\section{Further discussion} 

\label{disc_sect}

1. In Kinzebulatov-Sem\"{e}nov \cite{KiS4}, the authors show that applying the Lions variation approach for $\partial_t-\Delta + b \cdot \nabla$  in the Bessel space $\mathcal W^{1/2,2}$ rather than $L^2$ allows to enlarge the class of admissible vector fields from the classical form-bounded vector fields $\mathbf{F}_\delta$ to the weakly form-bounded vector fields $\mathbf{F}_\delta^{\scriptscriptstyle 1/2} \subset \mathbf{MF}_\delta$. (In fact, the class $\mathbf{F}_\delta$ is dictated by the Lions approach ran in $L^2$.) Hence one obtains  existence and uniqueness of weak solution to Cauchy problem for $\partial_t-\Delta + b \cdot \nabla$, $b \in \mathbf{F}_\delta^{\scriptscriptstyle 1/2}$ in $\mathcal W^{1/2,2}$. This result does not impose any assumptions on ${\rm div\,}b$, but requires $\delta<1$. Since having a divergence-free $b$  one expects to have no constraints of the value of $\delta$ except that it is finite, this result does not settle the question of a posteriori theory for \eqref{kolm}, even for $a=I$ and weakly-form bounded  $b$.

\smallskip

2. Requiring $b \in \mathbf{F}_\delta^{\scriptscriptstyle 1/2}$ with $\delta$ sufficiently small (without any assumptions on ${\rm div\,}b$) yields $\mathcal W^{1+\frac{1}{p}-,p}$-regularity theory of $\partial_t-\Delta + b \cdot \nabla$, with the interval of admissible $p$ expanding to $]1,\infty[$ as $\delta \downarrow 0$, see \cite{Ki, KiS1}. 

\smallskip

3. In absence of any assumptions on ${\rm div\,}b$, De Giorgi's method yields the Harnack inequality for \eqref{eq2} when $b \in \mathbf{F}_\delta$, $\delta<\sigma^2$. In view of the previous comment, one can ask if De Giorgi's method also works for $a=I$ and $b \in \mathbf{F}_\delta^{\scriptscriptstyle 1/2}$ with weak form-bound $\delta<1$.
One obstacle when working directly with $\mathbf{F}_\delta^{\scriptscriptstyle 1/2}$ is the need to handle non-local operators. Interestingly, a larger class $\mathbf{MF}_\delta$ allows one to stay in the local setting at expense of imposing additional assumptions on ${\rm div\,}b$. (One practical outcome of this is that when  one approximates $b$ by bounded smooth vector fields $b_n$, e.g.\,in the proofs of Theorems \ref{thm2} and \ref{thm3}, it is easier to control simultaneously the multiplicative form-bound of $b_n$ and the form-bounds of ${\rm div\,}b_{n,\pm}$, than to control the weak form-bound of $b_n$ and the form-bound of ${\rm div\,}b_{n,\pm}$.)

\smallskip

4. Both classes $\mathbf{MF}_\delta$ and $\mathbf{ BMO}^{-1}$ are contained in a larger class: $b \in [\mathcal S']^d$ such that
\begin{equation}
\label{mf_2}
|\langle b\varphi,\varphi\rangle| \leq \delta\|\nabla \varphi\|_2\|\varphi\|_2 + c_\delta\|\varphi\|_2^2 \quad \forall\,\varphi \in C_c^\infty.
\end{equation}
This class was considered in \cite{KiS4} where it was proved  that  \eqref{mf_2}, together with the hypothesis ``${\rm div\,}b$ in the Kato class of potentials with sufficiently small Kato norm'', provides a priori Gaussian upper bound on the heat kernel of \eqref{kolm}; an a priori Gaussian lower bound in \cite{KiS4} is proved  under somewhat stronger assumption \eqref{bb}.

\smallskip

5. There is an analogy between the approximation uniqueness for Dirichlet problem for \eqref{eq2}, discussed in Theorem \ref{thm3}, and the uniqueness of ``good solution'' to Dirichlet problem for non-divergence form elliptic equations studied by Krylov, Safonov and Nadirashvili among others, see discussion in \cite{Sa}. The analogy is not just formal: being able to treat a large class of drifts allows one to put non-divergence form equations in divergence form (this was exploited e.g.\,in \cite{KiS2} in the study of SDEs with diffusion coefficients critical discontinuities and form-bounded drifts.)

\smallskip

6. The iteration procedure used in the proof of Caccioppoli's inequality in Proposition \ref{c_prop} also works for the corresponding parabolic equation $\partial_t-\nabla \cdot a \cdot \nabla + b \cdot \nabla=0$ where the class $\mathbf{MF}_\delta$ is now defined as the class of time-inhomogeneous vector fields 
$b \in [L^1_{\loc}(\mathbb R_+ \times \mathbb R^d)]^d$ such that for a.e.\,$t \in \mathbb R_+$,
\begin{equation*}
\langle |b(t)|\varphi,\varphi\rangle   \leq \delta \|\nabla \varphi\|_2\|\varphi\|_2 + c_\delta\|\varphi\|_2^2, \quad\forall\, \varphi \in  W^{1,2}.
\end{equation*}
(furthermore, constant $c_\delta$ can be replaced by a function of time). We are interested, in particular, in applications to weak well-posedness of SDEs, which require regularity estimates on solution to Cauchy problem in $\mathbb R^d$
$$
(\partial_t  - \Delta  + b \cdot \nabla) u=|\mathsf{f}|g, \quad u(0)=0,
$$
where $\mathsf{f} \in \mathbf{MF}_\mu$, $g \in C_c^2$,
cf.\,\cite{KiM, KiS5} for details. The proof of such estimates for multiplicatively form-bounded $b$, $\mathsf{f}$ presents its own set of difficulties, which we plan to address elsewhere.


\begin{thebibliography}{99}

\bibitem[A]{A} D. Adams, Weighted nonlinear potential theory, {\em Trans. Amer. Math. Soc.}
\textbf{297} (1986), 73-94.

\bibitem[CWW]{CWW} S.Y.A.\;Chang, J.M.\;Wilson and T.H.\;Wolff, Some weighted norm inequalities concerning
the Schr\"{o}dinger operator, \newblock { \em Comment.\;Math.\;Helvetici}, \textbf{60} (1985), 217-246.


\bibitem[CLMS]{CLMS} R. Coifman, P.-L. Lions, Y. Meyer and S. Semmes, Compensated compactness and Hardy spaces,
{\em J. Math. Pures Appl.} \textbf{72} (1992), 247-286.


\bibitem[FV]{FV} S.\,Friedlander and V.\,Vicol, Global well-posedness for an advection-diffusion equation arising in
magneto-geostrophic dynamics, {\em Ann. Inst. H. Poincar\'{e} Anal. Non Lin\'{e}aire} \textbf{28}(2) (2011), 283-301.

\bibitem[G]{G} E.\,Giusti, Direct Methods in the Calculus of Variations, {\em World Scientific}, 2003.


\bibitem[F]{F} N.\,Filonov, On the regularity of solutions to the equation $-\Delta u + b \cdot \nabla u=0$, {\em J.\,Math.\,Sciences}, \textbf{195}(1) (2013), 98-108 (published in Zapiski Nauchnykh Seminarov POMI, \textbf{410} (2013), 168-186).

\bibitem[H]{H} T.\,Hara, A refined subsolution estimate of weak subsolutions to second order
linear elliptic equations with a singular vector field, {\em Tokyo J.\,Math.}, \textbf{38}(1) (2015), 75-98.

\bibitem[Ki]{Ki} D.\,Kinzebulatov, A new approach to the $L^p$-theory of $-\Delta + b\cdot\nabla$, and its applications to Feller processes with general drifts,
\newblock {\em Ann.~Scuola~Norm.~Sup.~Pisa Cl. Sci. (5)}, \textbf{17} (2017), 685-711. 

\bibitem[KiM]{KiM} D.\,Kinzebulatov and K.R.\,Madou, Stochastic equations with time-dependent singular drift, {\em J.\,Differential Equations}, to appear.


\bibitem[KiS1]{KiS1} D.\,Kinzebulatov and Yu.\,A.\,Sem\"{e}nov, On the theory of the Kolmogorov operator in the spaces $L^p$ and $C_\infty$, {\em Ann.~Scuola~Norm.~Sup.~Pisa Cl. Sci. (5)} \textbf{21} (2020), 1573-1647.

\bibitem[KiS2]{KiS2} D.\,Kinzebulatov and Yu.A.\,Sem\"{e}nov, Feller generators and stochastic differential equations with singular (form-bounded) drift, \newblock{\em Osaka J.\,Math.}, \textbf{58} (2021), 855-883.


\bibitem[KiS3]{KiS3} D.\,Kinzebulatov and Yu.\,A.\,Sem\"{e}nov, Heat kernel bounds for parabolic equations with singular (form-bounded) vector fields, {\em Math.\,Ann.}, to appear.

\bibitem[KiS4]{KiS4} D.\,Kinzebulatov and Yu.\,A.\,Sem\"{e}nov, Regularity for parabolic equations with singular non-zero divergence vector fields, {\em Preprint}, arXiv:2205.05169 (2022).

\bibitem[KiS5]{KiS5} D.\,Kinzebulatov and Yu.\,A.\,Sem\"{e}nov, {\em Sharp solvability for singular SDEs}, Preprint, arXiv:2110.11232.


\bibitem[KT]{KT} H.\,Koch and D.\,Tataru, Well-posedness for the Navier-Stokes equations, {\em Adv.
Math.} \textbf{157} (2001), 22-35.



\bibitem[KS]{KS} V.\,F.\,Kovalenko and Yu.\,A.\,Sem\"{e}nov,
{\newblock $C_0$-semigroups in $L^p(\mathbb R^d)$ and $C_\infty(\mathbb R^d)$ spaces generated by differential expression $\Delta+b\cdot\nabla$.} 
(Russian) {\em Teor. Veroyatnost. i Primenen.}, \textbf{35} (1990), 449-458; translation in {\em Theory Probab. Appl.} \textbf{35} (1990), 443-453.


\bibitem[MK]{MK}
A.\,J.\,Majda and P.\,R.\,Kramer, Simplified models for turbulent diffusion:
theory, numerical modelling, and physical phenomena, {\em Physics Reports} \textbf{314} (1999), 237-574. 


\bibitem[NU]{NU} A.\,F.\,Nazarov and N.\,N.\,Uraltseva, The Harnack inequality and related properties for solutions to elliptic and parabolic equations with divergence-free lower order coefficients, {\em Algebra i Analiz}, \textbf{23} (2011), 136-168.



\bibitem[O]{O} H.\,Osada, Diffusion processes with generators of generalized divergence form, {\em J. Math. Kyoto Univ.},
\textbf{27} (1987), 597-619.


\bibitem[QX]{QX} Z.\,Qian, G.\,Xi, Parabolic equations with singular divergence-free drift vector fields, {\em J. London Math. Soc.} \textbf{100} (1) (2019), 17-40.


\bibitem[Sa]{Sa} M.\,V.\,Safonov, Nonuniqueness for second-order elliptic equations with measurable coefficients, {\em SIAM J.\,Math.\,Anal.}, \textbf{30} (1999), 879-895.

\bibitem[S]{S} Yu.\,A.\,Sem\"{e}nov, \newblock Regularity theorems for parabolic equations, \newblock {\em J.\,Funct.\,Anal.}, \textbf{231} (2006), 375-417.

\bibitem[SSSZ]{SSSZ} G.\,Seregin, L.\,Silvestre, V.\,\v{S}verak and A.\,Zlato\v{s}, On divergence-free drifts, {\em J. Differential Equations}, \textbf{252}(1) (2012), 505-540.

\bibitem[Z]{Z} Q.\,S.\,Zhang, A strong regularity result for parabolic equations, {\em Comm.\,Math.\,Phys.} \textbf{244} (2004) 245-260.



\bibitem[Zh]{Zh} V.\,V.\,Zhikov, Remarks on the uniqueness of a solution
of the Dirichlet problem for second-order elliptic
equations with lower-order terms, {\em Funktsional. Anal.
i Prilozhen.}, \textbf{38} (2004), 15-28.


\end{thebibliography}
\end{document}